\documentclass{amsart}
\pdfoutput=1
\usepackage{amsthm, amsmath,amsfonts,amscd,amssymb,graphicx}
\usepackage[all,cmtip]{xy}

\usepackage{fullpage}
\usepackage{verbatim}
\usepackage{url}
\usepackage{hyperref}

\usepackage{ float}

\usepackage{tikz,mathrsfs}
\usetikzlibrary{arrows,decorations.pathmorphing,decorations.pathreplacing,positioning,shapes.geometric,shapes.misc,decorations.markings,decorations.fractals,calc,patterns}

\tikzset{>=stealth',
        cvertex/.style={circle,draw=black,inner sep=1pt,outer sep=3pt},
        vertex/.style={circle,fill=black,inner sep=1pt,outer sep=3pt},
        star/.style={circle,fill=yellow,inner sep=0.75pt,outer sep=0.75pt},
        tvertex/.style={inner sep=1pt,font=\scriptsize},
        gap/.style={inner sep=0.5pt,fill=white}}

\def\qed{\hfill {\hbox{${\vcenter{\vbox{               
   \hrule height 0.4pt\hbox{\vrule width 0.4pt height 6pt
   \kern5pt\vrule width 0.4pt}\hrule height 0.4pt}}}$}}}

\numberwithin{equation}{section}
\theoremstyle{plain}
\newtheorem{theorem}[equation]{Theorem}
\newtheorem{proposition}[equation]{Proposition}
\newtheorem{corollary}[equation]{Corollary}
\newtheorem{lemma}[equation]{Lemma}

\theoremstyle{definition}

\newtheorem{defn}[equation]{Definition}

\newtheorem{example}[equation]{Example}

\newtheorem{remark}[equation]{Remark}

\newcommand{\rst}[1]{\ensuremath{{\mathbin\vert}%
\raise-.5ex\hbox{$#1$}}}

\newcommand\restr[2]{{
  \left.\kern-\nulldelimiterspace 
  #1 
  \vphantom{\big|} 
  \right|_{#2} 
  }}

\newcommand{\DMO}{\DeclareMathOperator}

\newcommand{\beq}{\begin{equation}}
\newcommand{\eeq}{\end{equation}}

\newcommand{\ang}[1]{\langle #1 \rangle}

\newcommand{\blank}{\mbox{$\underline{\makebox[10pt]{}}$}}

\newcommand{\st}{\, : \, }

\providecommand{\abs}[1]{\lvert#1\rvert}

\newcommand{\ssm}{\smallsetminus}

\newcommand{\hra}{\hookrightarrow}

\DMO{\id}{{Id}}
\DMO{\Hom}{{Hom}}

\DMO{\End}{{End}}
\DMO{\Ext}{{Ext}}
\DMO{\ext}{{ext}}
\DMO{\Tor}{Tor}
\DMO{\Aut}{{Aut}}
\DMO{\Proj}{Proj}
\DMO{\Spec}{Spec}
\DMO{\codim}{codim}
\DMO{\reg}{reg}
\DMO{\len}{len}
\DMO{\red}{red}
\DMO{\pd}{pd}
\DMO{\hd}{hd}
\DMO{\Ker}{Ker}
\DMO{\coker}{Coker}
\DMO{\im}{Im}
\DMO{\cohdim}{cd}
\DMO{\chrr}{char}
\DMO{\Skl}{Skl}
\DMO{\Inn}{Inn}
\DMO{\GK}{GKdim}
\DMO{\gr}{gr}
\DMO{\Div}{Div}
\DMO{\Bs}{Bs}
\DMO{\sing}{sing}
\DMO{\op}{op}
\DMO{\Ann}{Ann}
\DMO{\sat}{sat}
\DMO{\grade}{grade}
\DMO{\GKdim}{GKdim}
\DMO{\hilb}{hilb}
\DMO{\gldim}{gldim}
\DMO{\Tot}{Tot}
\DMO{\Pic}{Pic}

\newcommand{\mc}{\mathcal}
\newcommand{\mf}{\mathfrak}
\newcommand{\mb}{\mathbb}

\newcommand{\FF}{\mb F}

\newcommand{\NN}{{\mb N}}
\newcommand{\PP}{{\mb P}}

\newcommand{\RR}{\mb R}
\newcommand{\ZZ}{{\mb Z}}
 
\DMO{\shHom}{\mathcal{H}\!\mathit{om}}
\DMO{\shExt}{\mathcal{E}\!\mathit{xt}}
\DMO{\shTor}{\mathcal{T}\!\mathit{or}}

\newcommand{\sO}{\mc{O}}

\DMO{\rmod}{mod-\!}
\DMO{\lmod}{\!-mod}
\DMO{\lMod}{\!-Mod}
\DMO{\rMod}{Mod-\!}
\DMO{\lQgr}{\!-Qgr}
\DMO{\QCoh}{Qcoh}
\DMO{\Xyz}{Xyz}
\DMO{\xyz}{xyz}

\newcommand{\V}[2]{{
\mathbf V_{\boldsymbol{#1}}^{#2}}}

\newcommand{\Qgr}[1]{{
\operatorname{Qgr}({#1})}}

\newcommand{\Gr}[1]{{
\operatorname{Gr}({#1})}}

\newcommand{\Fdim}[1]{{
\operatorname{Fdim}({#1})}}

\DMO{\Irr}{Irr}

\newcommand{\X}{{\mathsf{X}}}

\date{\today}

\title{Path algebras of quivers and representations of locally
finite Lie algebras} 

\author{J. M. Hennig and S. J. Sierra} 
\address{(Hennig) Department of Mathematical and Statistical Sciences, 
University of Alberta,
Edmonton, Alberta 
Canada T6G 2G1}
\email{jhennig1@ualberta.ca}
 \address{(Sierra) School of Mathematics,
University of Edinburgh, 
James Clerk Maxwell Building, 
The King's Buildings, 
Peter Guthrie Tait Road, 
Edinburgh EH9 3FD, UK.}
\email{s.sierra@ed.ac.uk}
\thanks{The second author is partially supported by EPSRC grant EP/M008460/1.}

 

\begin{document}


\maketitle

\begin{abstract}
We explore the (noncommutative) geometry of locally simple representations of the diagonal locally finite Lie algebras $\mf{sl}(n^\infty)$, $\mf o(n^\infty)$, and $\mf{sp}(n^\infty)$. 
 Let $\mf g_\infty$ be one of these Lie algebras, and let    $I \subseteq U(\mf g_\infty)$ be the nonzero annihilator of a locally simple $\mf g_\infty$-module.  
We show that for each such  $I$, there is a quiver $Q$ so that locally simple $\mf g_\infty$-modules with annihilator $I$ are parameterised by ``points'' in the ``noncommutative space'' corresponding to the path algebra of  $ Q$.  
Methods of noncommutative algebraic geometry are key to this correspondence.
We classify the quivers that arise and relate them to characters of symmetric groups.
\end{abstract}

\tableofcontents

\smallskip

\section{Introduction}\label{INTRO}

Throughout, we work over a fixed  algebraically closed field $\FF$ of characteristic 0. 

We begin with an  example. 
The Lie algebra $\mf{sl}(n^\infty)$ is the direct limit
\[ \mf{sl}(n) \to \mf{sl}(n^2) \to \mf{sl}(n^3) \to \cdots, \]
where the maps at each stage send $M$ to 
$\operatorname{diag}(\underbrace{M, \dots, M}_{n})$.
A {\em natural representation} of $\mf{sl}(n^\infty)$ is a direct limit of the natural representations of each $\mf{sl}(n^k)$.
It is a direct limit of finite dimensional simple representations and so is {\em locally simple}.

For the Lie algebra $\mf{sl}(\infty)$, defined as the limit of the ``top left corner'' maps
\[ \mf{sl}(k) \to \mf{sl}(k+1) \to \mf{sl}(k+2) \to \cdots ,\]
there is a unique natural representation, which is $\FF^\infty$, regarded as an infinite column \cite{PSe1}.
However, for $\mf{sl}(n^\infty)$ the situation is more interesting.  In fact, we have:
\begin{proposition}\label{iprop:pointseq}
 Let $n \in \ZZ_{n \geq 2}$ and let $W = \FF^n$.  Natural representations of $\mf{sl}(n^\infty)$ are parameterised by tails of sequences of points $p_1, p_2, \dots \in \PP(W)$ --- that is, by equivalence classes of sequences of points under the relation $(p_\bullet) \sim (q_\bullet)$ if $p_k = q_k$ for $k \gg 0$.
\end{proposition}

In noncommutative algebraic geometry, it is well-known that tails of point sequences in $\PP(W)$ parameterise point modules over the tensor (free) algebra $T(W^*)$.  
Recall that if $R$ is an $\NN$-graded $\FF$-algebra, a {\em point module} for $R$ is a graded cyclic $R$-module $M = \bigoplus_{k \geq 0} M_k$  with $\dim M_k = 1$ for all $k$.
Let $\Qgr{R}$ be the quotient category 
\[ \{ \mbox{ graded left $R$-modules }\} / \{ \mbox{ finite dimensional modules }\}.\]
The precise statement is that tails of point sequences in $\PP(W)$ parameterise isomorphism classes of point modules in $\Qgr{T(W^*)}$.

Proposition~\ref{iprop:pointseq} and the discussion above suggest a functorial relationship between  $\Qgr{T(W^*)}$ and a subcategory of $\operatorname{Rep} (\mf{sl}(n^\infty)) = U(\mf{sl}(n^\infty)) \lMod$.  
We show that there is such a relationship, and generalize it to other locally simple representations.
Our results apply to representations of the Lie algebras $\mf{sp}(n^\infty)$ and $\mf{so}(n^\infty)$ as well (these Lie algebras are defined similarly to $\mf{sl}(n^\infty)$).  

Our main theorem is:
\begin{theorem}\label{ithm:catequiv}
 Let $\mf g_\infty$ be one of the Lie algebras $\mf{sl}(n^\infty)$, $\mf{sp}(n^\infty)$, or $\mf{so}(n^\infty)$, and let $I$ be a nonzero annihilator of a locally simple $\mf g_\infty$-module.
Then there is a finite quiver $Q$ so that the categories $U(\mf g_\infty)/I \lMod$ and $\Qgr{\FF Q}$ are equivalent, where $\FF Q$ is the path algebra of $Q$ with the path length grading.
This category is also equivalent to the graded module category of the Leavitt path algebra of $Q$.  Moreover, under this equivalence locally simple representations of $\mf g_\infty$ with annihilator $I$ are in bijection with point modules in $\Qgr{\FF Q}$.
\end{theorem}

If $\mf g_\infty = \mf{sl}(n^\infty)$ and $I$ is the annihilator of a natural representation, then $Q$ is the 1-vertex $n$-loop quiver 
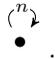
\[ 
\begin{array}{c}
\begin{tikzpicture}[bend angle=20, looseness=1]
\node (a) at (-1,0) {$\bullet$};
 \draw[->]  (a) edge [in=55,out=120,loop,looseness=4] node[gap] {$\scriptstyle n$}  (a);
\end{tikzpicture}.
\end{array}
\]
Thus $\FF Q \cong \FF\ang{x_1, \dots, x_n} \cong T(W^*)$.  Since all natural representations have the same annihilator, we recover Proposition~\ref{iprop:pointseq}.
In the body of the paper, we also give a version of Proposition~\ref{iprop:pointseq} for more general locally simple representations.

For each of the Lie algebras above, we give explicit formul\ae \ for the quivers $Q$ which occur; they come in finitely many infinite families.
There is a beautiful combinatorics associated to one of these infinite families, which we call the ``Type I" quivers. These quivers arise from branching laws involving Littlewood-Richardson coefficients, and can be interpreted entirely in terms of the symmetric group $S_n$. Using induction and restriction functors in the representation theory of $S_n$, we are able to provide  
a closed-form description of the Type I quivers in terms of the character table of the symmetric group $S_n$.

For example, one of the Type I quivers for $\mathfrak{sl}(n^{\infty})$ is the following:
\beq\label{iquiver}
\begin{tikzpicture}[bend angle=20, looseness=1]
\node (a) at (-1,0) {$(2)$};
\node (b) at (1,0)  {$(1,1)$};
\draw[->,bend left] (b) to node[gap] {$\scriptstyle a$} (a);
\draw[->,bend left] (a) to node[gap] {$\scriptstyle a$} (b);
\ \draw[->]  (a) edge [in=55,out=120,loop,looseness=4] node[gap] {$\scriptstyle b$}  (a);
 \draw[->]  (b) edge [in=55,out=120,loop,looseness=4] node[gap] {$\scriptstyle b$}  (b);
\end{tikzpicture} 
, \quad \mbox{where $a=\binom{n-1}{2}$, $b = \binom{n+1}{2}$.}
\eeq
The adjacency matrix of this quiver is 
\[ 
\begin{pmatrix} \frac{n^2+n}{2} & \frac{n^2-n}{2} \\ \frac{n^2-n}{2} & \frac{n^2+n}{2} \end{pmatrix}
=
\begin{pmatrix} 1& 1 \\ 1 & -1 \end{pmatrix}
\begin{pmatrix} n^2 &  0 \\ 0 & n \end{pmatrix}
\begin{pmatrix} 1& 1 \\ 1 & -1 \end{pmatrix}^{-1}.
\]
Note that $ \left( \begin{array}{cc} 1 & 1 \\ 1 & -1 \end{array} \right)$  is the character table for $S_2$.

We use the Type I quivers to give a new proof of simplicity of many factors of $U(\mf{sl}(n^\infty))$. We also show that the remaining quivers for $\mathfrak{sl}(n^{\infty})$, as well as the quivers associated to $\mathfrak{sp}(n^{\infty})$ and $\mathfrak{so}(n^{\infty})$, can be described in terms of the Type I quivers. 

We were initially motivated to study diagonal locally finite Lie algebras, as opposed to the finitary Lie algebras such as $\mf{sl}(\infty)$, because diagonal Lie algebras have infinitely many non-isomorphic natural representations which are parameterized by a noncommutative space, while $\mf{sl}(\infty)$ has only one natural representation, as explained above. Our work relies on results of Penkov and Petukhov in \cite{PP}, which shows (counter to intuition) that the ideal structure of the universal enveloping algebras of diagonal Lie algebras is simpler to understand than in the finitary case. We explain the relevant results in Section \ref{BACKGROUND}. Therefore, while the ideal lattice of the universal enveloping algebra $U(\mathfrak{g}_{\infty})$ is simpler when $\mf{g}_{\infty}$ is diagonal and not finitary, the set of modules which have the same annihilator $I$, given by the category $U(\mf g_\infty)/I \lMod$, is more interesting in terms of geometry.

The paper is organised as follows.  In Part~\ref{PART1} (Sections~2--5) we give background on locally finite Lie algebras and their representation theory.
We relate this representation theory to quivers in Section~\ref{QUIVER}.
In Part~\ref{PART2} (Sections~6--8) we construct category equivalences between representations of a locally finite Lie algebra and quiver representations.  We prove most of Theorem~\ref{ithm:catequiv} in Section~\ref{QGR} and generalise Proposition~\ref{iprop:pointseq} to complete the proof of Theorem~\ref{ithm:catequiv} in Section~\ref{POINTS}.
Finally, in Part~\ref{PART3} (Sections~9--11) we analyze the particular quivers that occur.

\vspace{5pt}

\noindent {\bf Acknowledgements:}  This project began with a visit by the first author to the University of Edinburgh in February 2014, which was funded by an AMS-Simons travel grant of the second author.  We thank them for their support.  
In addition, we thank Ken Goodearl, Ivan Penkov, Alexey Petukhov, and  Paul Smith for several useful conversations. We particularly thank David Speyer for recognising the relationship between eigenvectors of Type I matrices and characters of $S_n$.  

When this paper was in preparation, we learnt that Cody Holdaway had independently obtained some of our results on point spaces in his Ph.D. thesis.  We thank him for sharing his unpublished manuscript \cite{Ho} with us.

We thank the referees of our manuscript for their careful reading and helpful comments. 

In addition, we would like to thank William Stein for his tireless work over many years to develop the open-source Sage mathematical software.  The matrices in Section~\ref{APPENDIX} were computed in SageMathCloud, which greatly aided in the discovery of the combinatorial results in Part~\ref{PART3}.

\part{Irreducible representations of diagonal locally finite Lie algebras}\label{PART1}

\section{Background}\label{BACKGROUND}

A Lie algebra is {\em locally finite dimensional}, or {\em locally finite}, if every finitely generated subalgebra is finite dimensional. A locally finite Lie algebra $L$ is called {\em locally simple} if $L$ is the direct limit of finite dimensional simple subalgebras. The Lie algebras that we will consider are all locally simple. In this section, $\mathfrak{g}_{\infty} = \varinjlim \mathfrak{g}_k$ will denote a locally simple Lie algebra. 
Note that a locally simple Lie algebra is automatically simple.

A key tool in the representation theory of locally finite Lie algebras is Zhilinskii's technique of coherent local systems.  
A \textit{coherent local system of modules} (shortened as \textit{c.l.s.}) for $\mathfrak{g}_{\infty} = \varinjlim \mathfrak{g}_k$ is a collection of sets $\{ C_k \}_{k \in \mathbb{N}}$ where each $C_k$ consists of isomorphism classes of finite dimensional irreducible modules of $\mathfrak{g}_k$.  We require that  a c.l.s. is stable under restriction:  that is, 
the modules in $C_k$ are the irreducible summands of the restrictions of the  modules in $C_l$ to $\mathfrak{g}_k$ for each $l \ge k$. 

If $C$ is a c.l.s., then $\cap_{z \in C_k} \Ann_{U(\mathfrak{g}_k)}(z) \subseteq \cap_{z \in C_l} \Ann_{U(\mathfrak{g}_l)}(z)$ for any $l \ge k$. Therefore setting $I_k = \cap_{z \in C_k} \Ann_{U(\mathfrak{g}_k)}(z)$, we obtain an ideal associated to C by setting $I(C) = \cup_k I_k$. This ideal, if non-zero, will be \textit{integrable}, meaning for any $k$, $I \cap U(\mathfrak{g}_k) = I_k$ is the intersection of ideals of finite codimension in $U(\mathfrak{g}_k)$. 

For any ideal $I \subseteq U(\mathfrak{g}_{\infty})$, we can associate to $I$ a coherent local system $C(I)$ by setting
\[ C(I)_k = \{ z \in \Irr(\mathfrak{g}_k) \st I \cap U(\mathfrak{g}_k) \subseteq \Ann_{U(\mathfrak{g}_k)}(z)\}. \]

Note that if $I$ is integrable, then $I$ is the annihilator of $C(I)$, that is, $I(C(I)) = I$ (\cite{PP}). A c.l.s. $C$ is of \textit{finite type} if $|C_k| < \infty$ for all $k$. If $C$ has finite type, then $I(C)$ is of locally finite codimension, meaning $I \cap U(\mathfrak{g}_k) = I_k$ is of finite codimension in $U(\mathfrak{g}_k)$ for all $k$. Any ideal of locally finite codimension is automatically integrable. 

\begin{proposition}\label{1to1}
There is a 1-1 correspondence between non-zero ideals $I \subseteq U(\mathfrak{g}_{\infty})$ of locally finite codimension and coherent local systems $C$ of finite type.
\end{proposition}

\begin{proof}
Suppose we begin with a c.l.s. $C$. We wish to show that $C(I(C)) = C$. Set $I = I(C)$ and recall that $I_k = \cap_{z \in C_k} \Ann_{U(\mathfrak{g}_k)}(z)$. For each $z \in C_k$, we have that $I_k \subseteq \Ann_{U(\mathfrak{g}_k)}(z)$, thus $C_k \subseteq C(I)_k$. On the other hand, suppose $y \in C(I)_k$. Then $I_k = \cap_{z \in C_k} \Ann_{U(\mathfrak{g}_k)}(z) \subseteq \Ann_{U(\mathfrak{g}_k)}(y)$. Since $\Ann_{U(\mathfrak{g}_k)}(y)$ is a primitive, hence prime, ideal of $U(\mathfrak{g}_k)$, we have that $\Ann_{U(\mathfrak{g}_k)}(z) \subseteq \Ann_{U(\mathfrak{g}_k)}(y)$ for some $z \in C_k$. But this implies $\Ann_{U(\mathfrak{g}_k)}(z) = \Ann_{U(\mathfrak{g}_k)}(y)$, since $\Ann_{U(\mathfrak{g}_k)}(z)$ and $\Ann_{U(\mathfrak{g}_k)}(y)$ are both maximal ideals.
Since each primitive ideal is the annihilator of a unique finite dimensional irreducible $\mathfrak{g}_k-$module, we have that $y = z \in C_k$.
\end{proof}

This correspondence is reminiscent of the  correspondence in algebraic geometry between algebraic sets and radical ideals. Note that if $C$ is the empty c.l.s., that is, $C_k = \emptyset$ for all $k$, then $I(C) = U(\mathfrak{g}_{\infty})$. 

The representation theory of locally finite Lie algebras is quite complicated, and it is natural to restrict one's attention to integrable modules. A module $M$ over $\mathfrak{g}_{\infty}$ is \textit{integrable} if for any $m \in M$ and $x \in \mathfrak{g}_{\infty}$, dim Span$_{\mathbb{F}}\{ m, xm, x^2m, \dots \} < \infty$. Any integrable module determines a c.l.s. $C$ by setting 
\[C_k = \{ z \in \Irr(\mathfrak{g}_k) \st \Hom_{\mathfrak{g}_k}(z, \restr{M}{\mf g_k}) \ne 0\}. \]
An integrable module $M$ over $\mathfrak{g}_{\infty} = \varinjlim \mathfrak{g}_k$ is \textit{locally simple} if $M = \varinjlim M_k$, where each $M_k$ is a simple finite dimensional module over $\mathfrak{g}_k$; such a module is also simple.
These modules are related to coherent local systems which are irreducible: A coherent local system $C$ is \textit{irreducible} if $C$ cannot be written as $C' \cup C''$ where $C' \not \subseteq  C''$ and $C'' \not \subseteq C'$. The following result of Zhilinskii  \cite[Lemma~1.1.2 ]{Zh1} (see \cite[Proposition~7.3]{PP} for translation) relates irreducible coherent local systems to prime ideals in $U(\mathfrak{g}_{\infty})$.

\begin{lemma}{\rm (Zhilinskii)}\label{Zh1prime}
If $C$ is an irreducible coherent local system, then $I(C)$ is the annihilator of some locally simple integrable $\mathfrak{g}_{\infty}$-module. In particular, $I(C)$ is primitive and hence prime.
\end{lemma}

In \cite{Zh1}, Zhilinskii also showed that any c.l.s. $C$ can be written as a finite union $C = \bigcup_{i=1}^r C_i$, where each $C_i$ is maximal among the irreducible c.l.s's contained in $C$.  We say that  $C_1, \dots, C_r$  are the {\em irreducible components} of $C$. From this we have the following:

\begin{lemma} \label{lem:primeprim}
{\rm(\cite[Proposition 7.7]{PP})}
Let $\mathfrak{g}_{\infty}$ be a locally simple Lie algebra. An integrable ideal of $U(\mathfrak{g}_{\infty})$ is prime if and only if it is primitive.
\end{lemma} 
(Note that while this proposition is stated in \cite{PP} as holding only for the finitary case, the same proof works for an arbitrary locally simple Lie algebra.)

We now restrict our attention to {\em diagonal} locally simple Lie algebras. A locally simple Lie algebra $\mathfrak{g_{\infty}} = \varinjlim \mathfrak{g}_k$ is {\em diagonal} if each embedding $\mathfrak{g}_k \hookrightarrow \mathfrak{g}_{k+1}$ is a diagonal embedding, meaning when one considers $\mathfrak{g}_{k+1}$ as a module over $\mathfrak{g}_k$, the only possible irreducible submodules which appear are the standard $\mathfrak{g}_k$-module, its dual, and the trivial 1-dimensional $\mathfrak{g}_k$-module. For example, any diagonal embedding $\phi: \mathfrak{sl}_n \rightarrow \mathfrak{sl}_m$ is of the form:
\[ \phi(A) = diag(\underbrace{A, A, \dots, A}_{l}, \underbrace{-A^T, \dots, -A^T}_{r}, \underbrace{0, \dots, 0}_{z})  \]
where $(l+r)n + z = m$. The triple $(l,r,z)$ is called the {\em signature} of the embedding. Locally simple, diagonal Lie algebras have been classified by Baranov and Zhilinskii in \cite{BZh}. 

Much of the literature on representation theory of locally finite Lie algebras have focused on the three {\em finitary} Lie algebras $\mathfrak{sl}(\infty)$, $\mathfrak{sp}(\infty)$, and $\mathfrak{so}(\infty)$, which result from embeddings of signature $(1,0,1)$. In particular, Penkov, Serganova, and collaborators have explored integrable and tensor representations of finitary Lie algebras in \cite{DPS}, \cite{PSe1}, \cite{PSe2}, and \cite{PSt}.

In contrast, the Lie algebras we study will be $\mathfrak{g}_{\infty} = \mathfrak{sl}(n^{\infty}), \mathfrak{sp}(n^{\infty})$ or $\mathfrak{so}(n^{\infty})$, which all result from embeddings of signature $(n, 0 ,0)$, where $n \ge 2$ is a positive integer. 
(Note that $n$ must be even in the symplectic case.)
They are examples of diagonal, non-finitary Lie algebras. 
When $\mathfrak{g}_{\infty} = \mathfrak{sl}(n^{\infty}), \mathfrak{sp}(n^{\infty})$ or $\mathfrak{so}(n^{\infty})$,  it is shown in \cite[Corollary~3.2]{PP}  that every ideal of $U(\mathfrak{g}_{\infty})$ is of locally finite codimension and therefore integrable.  
It follows from Lemma~\ref{lem:primeprim} that every prime ideal of $U(\mathfrak{g}_{\infty})$ is primitive.

\begin{lemma} \label{lem:infinitetype}
Suppose $\mathfrak{g}_{\infty}$ is a diagonal, non-finitary Lie algebra. Suppose $C$ is a c.l.s. of infinite type. Then $I(C) = (0)$.
\end{lemma}

\begin{proof}
Suppose $I(C) \ne (0)$. Then by \cite[Corollary~3.2]{PP}, $I(C)$ is an ideal of locally finite codimension in $U(\mathfrak{g}_{\infty})$. By Proposition \ref{1to1}, $C(I(C)) = C$ is a c.l.s. of finite type, which is a contradiction. Therefore, $I(C) = (0)$.  
\end{proof}

We therefore have the following:

\begin{proposition}\label{prop:correspondence}
Suppose $\mathfrak{g}_{\infty}$ is a diagonal, non-finitary Lie algebra. There is a 1-1 correspondence between non-zero prime ideals of $U(\mathfrak{g}_{\infty})$ and irreducible coherent local systems of finite type.
\end{proposition}

\begin{proof}
If $C$ is an irreducible c.l.s. of finite type, then $I(C)$ is prime by Lemma \ref{Zh1prime}. If $I(C) = (0)$, then $C(I(C)) = C$ is the coherent local system with $C_k = \Irr(\mathfrak{g}_k)$, the full set of finite dimensional irreducible representations of $\mathfrak{g}_k$. This c.l.s. has infinite type, therefore $I(C) \ne (0)$.

On the other hand, suppose $I$ is a prime ideal. Then $I$  has locally finite codimension by \cite{PP} and  thus $C = C(I)$ is a c.l.s. of finite type. By \cite{Zh1}, we can write $C = C_1 \cup \dots \cup C_r$ for maximal, irreducible c.l.s. $C_i$. It follows that $I = I(C) \subseteq \cap_i I(C_i)$ and $I(C_1) \cdots I(C_r) \subseteq I(C) = I$. Since $I$ is prime, $I(C_i) \subseteq I$ for some $i$, in which case $I = I(C_i)$. Therefore, $C(I) = C(I(C_i)) = C_i$ is irreducible. This shows that the 1-1 correspondence in Proposition \ref{1to1} restricts to a 1-1 correspondence between non-zero prime ideals and irreducible c.l.s. of finite type.
\end{proof}

Thus, in order to classify the prime ideals of $U(\mathfrak{g}_{\infty})$, one must classify the irreducible coherent local systems of finite type. This was completed by Zhilinskii in \cite{Zh2}, and we describe the relevant results in Section \ref{CLS}.

To end this section, we recall some basic background from representation theory.
Recall (see \cite{FH}) that every finite dimensional irreducible $\mathfrak{gl}(k)$-module is determined by a nonincreasing sequence of $k$ complex numbers $\alpha_1 \ge \alpha_2 \ge \dots \ge \alpha_k$ where $\alpha_i - \alpha_{i+1} \in \mathbb{Z}$ for $i = 1, 2, \dots, k-1$. Because we are ultimately concerned with representations of $\mf{sl}(k)$, we will only consider highest weight modules where each $\alpha_i$ is an integer. We use $V^{k}_{\alpha}$ to denote this module, which is the unique finite dimensional irreducible module of $\mathfrak{gl}(k)$ with highest weight $\alpha = (\alpha_1, \dots, \alpha_k)$. Choose $p, q \ge 0$ with $p + q \le k$ such that $\alpha_i \ge 0$ for $1 \le i \le p$ and $\alpha_j \le 0$ for $k - q + 1 \le j \le k$. Set $\lambda_i = \alpha_i$ for $1 \le i \le p$ and $\mu_{k-j+1} = -\alpha_j$ for $k - q + 1 \le j \le k$ so $-\mu_1 = \alpha_k$ and $-\mu_q = \alpha_{k - q + 1}$. Then we can rewrite $\alpha = (\alpha_1, \dots, \alpha_k)=(\lambda_1, \lambda_2, \dots, \lambda_p, 0, \dots, 0, -\mu_q, \dots, -\mu_1)$. Therefore, each $\alpha$ is determined by a pair of partitions $\lambda$ and $\mu$ of length $p$ and $q$, respectively, where $\lambda: \lambda_1 \ge \lambda_2 \ge \dots \ge \lambda_p \ge 0$ and $\mu: \mu_1 \ge \mu_2 \ge \dots \ge \mu_q \ge 0$. We use $V^{k}_{\lambda, \mu}$ to denote the finite dimensional irreducible module of $\mathfrak{gl}(k)$ determined by the partitions $\lambda$ and $\mu$.  
By restriction, $V^{k}_{\alpha}$ and $V^{k}_{\lambda, \mu}$ will also be an irreducible highest weight representation of $\mf{sl}(k)$, although two different representations of $\mf{gl}(k)$ may restrict to isomorphic representations of $\mf{sl}(k)$: two $\mf{gl}(k)$-modules of highest weights $\alpha$ and $\beta$ will restrict to isomorphic $\mf{sl}(k)$-modules if and only if $\alpha_i - \beta_i$ is some constant independent of $i$.  

\begin{remark}\label{rem:natrep}
The Lie algebra $\mathfrak{sl}(k)$ has two non-isomorphic $k$-dimensional irreducible representations, each of which can be realized by considering $\mf{sl}(k)$ as a subset of $k \times k$ matrices. We fix the ``natural representation" to be the module given by the action of $\mf{sl}(k)$ via left multiplication on the set of column vectors of dimension $k$, which has highest weight corresponding to the partition $(1)$. \end{remark}

For 
$\mathfrak{sp}(2k)$, every finite dimensional irreducible module is determined by a nonincreasing sequence of nonnegative integers. 
We denote by $U^{2k}_{\lambda}$ the irreducible finite dimensional module of $\mathfrak{sp}(2k)$ of highest weight
 $\lambda = (\lambda_1, \lambda_2, \dots, \lambda_p, 0, \dots, 0)$, or equivalently, determined by the partition $\lambda: \lambda_1 \ge \lambda_2 \ge \dots \ge \lambda_p \ge 0$. 
 
 For $\mf{so}(k)$, the finite dimensional irreducible modules are determined by highest weight vectors which may be half-integers. However, there is a parameterization of finite dimensional irreducible modules over $\mf{so}(k)$ by partitions $\lambda$, where the sum of the first two columns of the Young diagram for $\lambda$ is at most $k$. We denote by $W^{k}_{\lambda}$ the irreducible finite dimensional module of $\mf{so}(k)$ determined by the partition $\lambda$ in this way, consistent with the notation of branching laws in \cite{HTW}.

We will sometimes use uniform notation for the three types of Lie algebras.  
Let $\mf g$ be one of $\mf{gl}, \mf{sl}, \mf{sp}, \mf{so}$ and let $\boldsymbol{\lambda}$ be an integral dominant weight for $\mf{g}(k)$; so if $\mf g = \mf{gl}$ then $\boldsymbol{\lambda} = (\lambda, \mu)$ for partitions $\lambda, \mu$.
By $\V{\lambda}{k}$, we denote the corresponding representation of $\mf{g}(k)$, which may be $U^{k}_\lambda$ or $W^{k}_\lambda$ if $\mf g= \mf{sp}$ or $\mf{so}$.

\section{Branching laws}\label{BRANCH}

To understand representations of our diagonal Lie algebras, it is important to understand the relationships between representations of $\mf g(n^k)$ and $\mf g(n^{k+1})$.
These are given by {\em  branching laws}, and in this section we give  laws which we will use for the rest of the paper. The material in this section is drawn from \cite{Hr, HTW}. Appendix A of \cite{Hr} provides a condensed summary of branching laws. 

Note that all of the branching laws for $\mathfrak{gl}(k)$, $\mf{sl}(k)$,  $\mathfrak{sp}(2k)$, and $\mathfrak{so}(k)$ only hold for $k$ sufficiently large; that is, we give {\em stable} branching laws. Precise conditions on $k$ are provided in \cite{HTW}, but for our purposes it will not matter, since we can always choose $k$ arbitrarily large.

We first give branching laws for $\mf{gl}(k)$.
Let $c^{\gamma}_{\alpha, \beta}$ denote the Littlewood-Richardson coefficient determined by partitions $\alpha, \beta$, and $\gamma$. 

\begin{proposition}\label{prop:branchgl2}
Let $\mathfrak{gl}(k) \hookrightarrow \mathfrak{gl}(2k)$ be an embedding of signature $(2, 0, 0)$ and fix partitions $\lambda, \mu$. Then for $k \gg 0$: 
\begin{eqnarray*}
\restr{V^{2k}_{\lambda, \mu}}{\mathfrak{gl}(k)}\cong \bigoplus\limits_{\substack{\alpha^+, \alpha^-, \beta^+, \beta^- \\ \lambda', \mu'}} c^{(\lambda, \mu)}_{(\alpha^+, \alpha^-), (\beta^+, \beta^-)}d^{(\lambda', \mu')}_{(\alpha^+, \alpha^-), (\beta^+, \beta^-)} V^k_{\lambda', \mu'}, 
\end{eqnarray*}
where
\begin{eqnarray*}
c^{(\lambda, \mu)}_{(\alpha^+, \alpha^-), (\beta^+, \beta^-)} = \sum_{\gamma^+, \gamma^-, \delta} c^{\lambda}_{\gamma^+, \delta}c^{\mu}_{\gamma^-, \delta}c^{\gamma^+}_{\alpha^+, \beta^+}c^{\gamma^-}_{\alpha^-, \beta^-} 
\end{eqnarray*}
and
\begin{eqnarray*}
d^{(\lambda', \mu')}_{(\alpha^+, \alpha^-), (\beta^+, \beta^-)} = \sum\limits_{\substack{\alpha_1, \alpha_2, \beta_1, \beta_2 \\ \gamma_1, \gamma_2}} c^{\alpha^+}_{\alpha_1, \gamma_1}c^{\beta^-}_{\gamma_1, \beta_2}c^{\alpha^-}_{\beta_1, \gamma_2}c^{\beta^+}_{\gamma_2, \alpha_2}c^{\lambda'}_{\alpha_1, \alpha_2}c^{\mu'}_{\beta_1, \beta_2}.
\end{eqnarray*}
\end{proposition}

\begin{proof}
Combine Theorems 2.1.1 and 2.2.1 from \cite{HTW}.
\end{proof}







It follows from properties of Littlewood-Richardson coefficients that if $V^k_{\lambda', \mu'}$ appears in this decomposition with nonzero multiplicity, then $|\lambda'| \le |\lambda|$, $|\mu'| \le |\mu|$, and $|\lambda| - |\mu| = |\lambda'| - |\mu'|$. 

Note that when $\mu = 0$, Proposition~\ref{prop:branchgl2} reduces to:
\begin{eqnarray}\label{branchgl20}
\restr{V^{2k}_{\lambda, 0}}{\mathfrak{gl}(k)} \cong \bigoplus_{\beta_1, \beta_2, \lambda'}  c^{\lambda}_{\beta_1, \beta_2}c^{\lambda'}_{\beta_1, \beta_2} V^k_{\lambda', 0}.
\end{eqnarray}
In this case, if $V^k_{\lambda', 0}$ appears with nonzero multiplicity, then  $|\lambda| = |\lambda'|$.

We now provide the general branching law for $\mathfrak{gl}(k) \hookrightarrow \mathfrak{gl}(nk)$.

\begin{proposition}\label{branchgln} 
{\rm (\cite[Proposition~2.4]{Hr})} 
Let $\mathfrak{gl}(k) \hookrightarrow \mathfrak{gl}(nk)$ be an embedding of signature $(n, 0, 0)$ where $n > 2$, and fix partitions $\lambda, \mu$.  For $k \gg 0$ we have:
\begin{eqnarray*}
\restr{V^{nk}_{\lambda, \mu}}{\mathfrak{gl}(k)}\cong \bigoplus\limits_{\substack{\beta_1^+, \beta_2^+, \dots, \beta_n^+ \\ \beta_1^-, \beta_2^-, \dots, \beta_n^- \\ \lambda', \mu'}} C^{(\lambda, \mu)}_{(\beta_1^+, \beta_2^+, \dots, \beta_n^+), (\beta_1^-, \beta_2^-, \dots, \beta_n^-)}D^{(\lambda', \mu')}_{(\beta_1^+, \beta_2^+, \dots, \beta_n^+), (\beta_1^-, \beta_2^-, \dots, \beta_n^-)} V^k_{\lambda', \mu'} 
\end{eqnarray*}
where
\begin{multline*}
C^{(\lambda, \mu)}_{(\beta_1^+, \beta_2^+, \dots, \beta_n^+), (\beta_1^-, \beta_2^-, \dots, \beta_n^-)} = \\
\sum\limits_{\substack{\alpha_1^+, \alpha_2^+, \dots, \alpha_{n-2}^+ \\ \alpha_1^-, \alpha_2^-, \dots, \alpha_{n-2}^-}} c^{(\lambda, \mu)}_{(\alpha_1^+, \alpha_1^-), (\beta_1^+, \beta_1^-)}c^{(\alpha_1^+, \alpha_1^-)}_{(\alpha_2^+, \alpha_2^-), (\beta_2^+, \beta_2^-)}\cdots c^{(\alpha_{n-3}^+, \alpha_{n-3}^-)}_{(\alpha_{n-2}^+, \alpha_{n-2}^-), (\beta_{n-2}^+, \beta_{n-2}^-)}c^{(\alpha_{n-2}^+, \alpha_{n-2}^-)}_{(\beta_{n-1}^+, \beta_{n-1}^-), (\beta_{n}^+, \beta_{n}^-)}
\end{multline*}
and
\begin{multline*}
D^{(\lambda', \mu')}_{(\beta_1^+, \beta_2^+, \dots, \beta_n^+), (\beta_1^-, \beta_2^-, \dots, \beta_n^-)} = \\
\sum\limits_{\substack{\alpha_1^+, \alpha_2^+, \dots, \alpha_{n-2}^+ \\ \alpha_1^-, \alpha_2^-, \dots, \alpha_{n-2}^-}} d^{(\alpha_1^+, \alpha_1^-)}_{(\beta_1^+, \beta_1^-),(\beta_2^+, \beta_2^-)}
d^{(\alpha_2^+, \alpha_2^-)}_{(\alpha_1^+, \alpha_1^-),(\beta_3^+, \beta_3^-)} \cdots 
d^{(\alpha_{n-2}^+, \alpha_{n-2}^-)}_{(\alpha_{n-3}^+, \alpha_{n-3}^-),(\beta_{n-1}^+, \beta_{n-1}^-)}
d^{\lambda', \mu'}_{(\alpha_{n-2}^+, \alpha_{n-2}^-), (\beta_{n}^+, \beta_{n}^-)}.
\end{multline*}
(Both $c^{(\lambda, \mu)}_{(\alpha^+, \alpha^-), (\beta^+, \beta^-)}$ and $d^{(\lambda', \mu')}_{(\alpha^+, \alpha^-), (\beta^+, \beta^-)} $ are defined in the statement of Proposition~\ref{prop:branchgl2}.)
\end{proposition}

If both $C^{(\lambda, \mu)}_{(\beta_1^+, \beta_2^+, \dots, \beta_n^+), (\beta_1^-, \beta_2^-, \dots, \beta_n^-)} \ne 0$ and $D^{(\lambda', \mu')}_{(\beta_1^+, \beta_2^+, \dots, \beta_n^+), (\beta_1^-, \beta_2^-, \dots, \beta_n^-)} \ne 0$, then:

\begin{eqnarray*}
|\lambda'| \le |\beta_1^+| + |\beta_2^+| + \dots + |\beta_n^+| \le |\lambda|, \\
|\mu'| \le |\beta_1^-| + |\beta_2^-| + \dots + |\beta_n^-| \le |\mu|.
\end{eqnarray*}
We can calculate that $|\lambda| - |\mu| = (|\beta_1^+| + |\beta_2^+| + \dots + |\beta_n^+|) - (|\beta_1^-| + |\beta_2^-| + \dots + |\beta_n^-|) = |\lambda'| - |\mu'|$. Therefore, if $V^k_{\lambda', \mu'}$ appears in the decomposition of $V^{nk}_{\lambda, \mu}$ with nonzero multiplicity, we again have that $|\lambda'| \le |\lambda|$, $|\mu'| \le |\mu|$, and $|\lambda| - |\mu| = |\lambda'| - |\mu'|$.

Note that when $\mu = 0$, Proposition~\ref{branchgln} reduces to the following well-known branching rule:
\begin{eqnarray}\label{branchgln0}
\restr{V^{nk}_{\lambda, 0}}{\mathfrak{gl}(k)} \cong \bigoplus_{\beta_1, \beta_2, \dots, \beta_n, \lambda'}  c^{\lambda}_{\beta_1, \beta_2, \dots, \beta_n}c^{\lambda'}_{\beta_1, \beta_2, \dots, \beta_n} V^k_{\lambda', 0} .
\end{eqnarray}
Here the coefficients $c^{\lambda}_{\beta_1, \beta_2, \dots, \beta_n}$ are the \textit{generalized} Littlewood Richardson coefficients, defined as:
\begin{eqnarray*}
c^{\lambda}_{\beta_1, \beta_2, \dots, \beta_n} = \sum_{\alpha_1, \dots, \alpha_{n-2}} c^{\lambda}_{\alpha_1, \beta_1}c^{\alpha_1}_{\alpha_2, \beta_2} \dots c^{\alpha_{n-3}}_{\alpha_{n-2},\beta_{n-2}}c^{\alpha_{n-2}}_{\beta_{n-1}, \beta_n}.
\end{eqnarray*}
In this case, if $V^k_{\lambda', 0}$ appears in the decomposition of $V^{nk}_{\lambda, 0}$ with nonzero multiplicity, then $|\lambda'| = |\lambda|$.

We have given branching laws for embeddings $\mf{gl}(n^k) \hra \mf{gl}(n^{k+1})$. 
We are more interested, however, in the embeddings $\mf{sl}(n^k) \hra \mf{sl}(n^{k+1})$. The following lemma states that embeddings from $\mf{sl}(n^k) \hra \mf{sl}(n^{k+1})$ will follow the same branching laws as those for $\mf{gl}$.

\begin{lemma}\label{lem:foobar}
 Let $\mf{sl}(k) \hra \mf{sl}(nk)$ be an embedding of signature $(n,0,0)$, and let $ \lambda, \mu$ be partitions.  For $k\gg 0$ the branching law given in Proposition~\ref{branchgln} (for $n>2$) or Proposition~\ref{prop:branchgl2} (for $n=2$) holds for $\restr{V^{nk}_{\lambda, \mu}}{\mf{sl}(k)}$.
\end{lemma}
\begin{proof}
We only need to prove that the distinct modules which appear with nonzero multiplicity in the decomposition of Proposition~\ref{branchgln} do not become isomorphic as $\mf{sl}(k)$-modules. Suppose $V^k_{\lambda', \mu'}$ and $V^k_{\lambda'', \mu''}$ appear with nonzero multiplicity, where $V^k_{\lambda', \mu'}$ is the $\mf{gl}(k)$-module of highest weight 
\[(\lambda'_1, \lambda'_2, \dots, \lambda'_{p'}, 0, \dots, 0, -\mu'_{q'}, \dots, -\mu'_1)\]
 and $V^k_{\lambda'', \mu''}$ is the $\mf{gl}(k)$-module of highest weight $(\lambda''_1, \lambda''_2, \dots, \lambda''_{p''}, 0, \dots, 0, -\mu''_{q''}, \dots, -\mu''_1)$. If $V^k_{\lambda', \mu'}$ and $V^k_{\lambda'', \mu''}$ are isomorphic as $\mf{sl}(k)$-modules, then $(\lambda''_1, \lambda''_2, \dots, \lambda''_{p''}, 0, \dots, 0, -\mu''_{q''}, \dots, -\mu''_1) = (\lambda'_1 + c, \lambda'_2 + c, \dots, \lambda'_{p'} + c, c, \dots, c, -\mu'_{q'} + c, \dots, -\mu'_1 + c)$ for some constant $c$. Then $|\lambda''| - |\mu''| = |\lambda'| - |\mu'| + kc$. However, since $V^k_{\lambda', \mu'}$ and $V^k_{\lambda'', \mu''}$ both appear in with nonzero multiplicity, we have $|\lambda''| - |\mu''| = |\lambda'| - |\mu'|$, and therefore $c = 0$. 
\end{proof}

We next derive analogous branching rules for $\mathfrak{sp}(n^\infty)$. 
Recall that if $\lambda$ is a partition, then  $U_{\lambda}^{2k}$ denotes the irreducible representation of $\mathfrak{sp}(2k)$ with highest weight $\lambda$. 
If $\delta$ is a partition, we denote its conjugate partition by $\delta^T$.

\begin{proposition}\label{branchsp2} 
Let $\mathfrak{sp}(2k) \hookrightarrow \mathfrak{sp}(4k)$ be an embedding of signature $(2,0,0)$, and fix a partition $\lambda$. For $k \gg 0$ we have:
\begin{eqnarray*}
\restr{U^{4k}_{\lambda}}{\mathfrak{sp}(2k)} \cong \bigoplus_{\mu, \nu, \lambda'} e^{\lambda}_{\mu, \nu} f^{\lambda'}_{\mu, \nu} U^{2k}_{\lambda'},
\end{eqnarray*}
where
\[ e^{\lambda}_{\mu, \nu} = \sum_{\delta, \gamma} c^{\gamma}_{\mu, \nu} c^{\lambda}_{\gamma, (2\delta)^{T}} \quad \mbox{and} \quad
 f^{\lambda'}_{\mu, \nu} = \sum_{\alpha, \beta, \gamma'} c^{\lambda'}_{\alpha, \beta} c^{\mu}_{\alpha, \gamma'} c^{\nu}_{\beta, \gamma'}. \]
\end{proposition}

\begin{proof}
Combine Theorems 2.1.3 and 2.2.3 from \cite{HTW}.
\end{proof}

From properties of Littlewood-Richardson coefficients, it follows that if $U^{2k}_{\lambda'}$ appears in the above decomposition of $U^{4k}_{\lambda}$ with nonzero multiplicity, then $|\lambda'| \le |\lambda|$ and $ |\lambda|$ and $|\lambda'|$ have the same parity. 

\begin{proposition}\label{branchspn}
Let $\mathfrak{sp}(2k) \hookrightarrow \mathfrak{sp}(2nk)$ be an embedding of signature $(n,0,0)$, and fix a partition $\lambda$. Then for $k \gg 0$:
\begin{eqnarray*}
\restr{U^{2nk}_{\lambda}}{\mathfrak{sp}(2k)} \cong \bigoplus_{\beta_1, \beta_2, \dots, \beta_n, \lambda'} E^{\lambda}_{\beta_1, \beta_2, \dots, \beta_n} F^{\lambda'}_{\beta_1, \beta_2, \dots, \beta_n} U^{2k}_{\lambda'}
\end{eqnarray*}
where 
\begin{eqnarray*}
E^{\lambda}_{\beta_1, \beta_2, \dots, \beta_n} = \sum_{\alpha_1, \alpha_2, \dots, \alpha_{n-2}} e^{\lambda}_{\alpha_1, \beta_1} e^{\alpha_1}_{\beta_2, \alpha_2} \cdots e^{\alpha_{n-3}}_{\beta_{n-2}, \alpha_{n-2}}e^{\alpha_{n-2}}_{\beta_{n-1}, \beta_n}
\end{eqnarray*}
and 
\begin{eqnarray*}
F^{\lambda'}_{\beta_1, \beta_2, \dots, \beta_n} = \sum_{\alpha_1, \alpha_2, \dots, \alpha_{n-2}} f^{\alpha_1}_{\beta_1, \beta_2} f^{\alpha_2}_{\alpha_1, \beta_3} \cdots f^{\alpha_{n-2}}_{\alpha_{n-3}, \beta_{n-1}} f^{\lambda'}_{\alpha_{n-2}, \beta_n}.
\end{eqnarray*}
(Both $e^{\lambda}_{\mu, \nu} $ and $f^{\lambda'}_{\mu, \nu} $ are defined in the statement of Proposition~\ref{branchsp2}.)
\end{proposition}

\begin{proof}
Iterate the branching law from Proposition \ref{branchsp2} using Theorem 2.1.3 and 2.2.3 from \cite{HTW}.  
\end{proof}

From properties of Littlewood-Richardson coefficients,  if $U^{2k}_{\lambda'}$ appears in the above decomposition of $U^{2nk}_{\lambda}$ with nonzero multiplicity, then $|\lambda'| \le |\lambda|$ and $|\lambda|$ and $|\lambda'|$ have the same parity.

We derive the branching laws for $\mathfrak{so}_k \hookrightarrow \mathfrak{so}_{nk}$ via similar calculations. 
Recall that $W_{\lambda}^k$ denotes the irreducible representation for $\mathfrak{so}(k)$ with highest weight $\lambda$. 

\begin{proposition}\label{branchso2}
Let $\mathfrak{so}_k \hookrightarrow \mathfrak{so}_{2k}$ be an embedding of signature $(2, 0, 0)$, and fix a partition $\lambda$. Then for $k \gg 0$:
\begin{eqnarray*}
\restr{W^{2k}_{\lambda}}{\mathfrak{so}_{k}} \cong \bigoplus_{\mu, \nu, \lambda'} g^{\lambda}_{\mu, \nu} f^{\lambda'}_{\mu, \nu} W^{k}_{\lambda'}
\end{eqnarray*}
where $f^{\lambda'}_{\mu, \nu}$ is defined as before and
\begin{eqnarray*}
g^{\lambda}_{\mu, \nu} = \sum_{\delta, \gamma} c^{\gamma}_{\mu, \nu} c^{\lambda}_{\gamma, (2\delta)}.
\end{eqnarray*}
\end{proposition}

\begin{proof}
Follows from Theorems 2.1.2 and 2.2.2 from \cite{HTW}.
\end{proof}

\begin{proposition}\label{branchson}
Let $\mathfrak{so}_k \hookrightarrow \mathfrak{so}_{nk}$ be an embedding of signature $(n, 0, 0)$, and fix a partition $\lambda$. Then for $k \gg 0$:
\begin{eqnarray*}
\restr{W^{nk}_{\lambda}}{\mathfrak{so}_{k}} \cong \bigoplus_{\beta_1, \beta_2, \dots, \beta_n, \lambda'} G^{\lambda}_{\beta_1, \beta_2, \dots, \beta_n} F^{\lambda'}_{\beta_1, \beta_2, \dots, \beta_n} W^{k}_{\lambda'}
\end{eqnarray*}
where $F^{\lambda'}_{\beta_1, \beta_2, \dots, \beta_n}$ is defined as before and 
\begin{eqnarray*}
G^{\lambda}_{\beta_1, \beta_2, \dots, \beta_n} = \sum_{\alpha_1, \alpha_2, \dots, \alpha_{n-2}} g^{\lambda}_{\alpha_1, \beta_1} g^{\alpha_1}_{\beta_2, \alpha_2} \cdots g^{\alpha_{n-3}}_{\beta_{n-2}, \alpha_{n-2}}g^{\alpha_{n-2}}_{\beta_{n-1}, \beta_n}.
\end{eqnarray*}
\end{proposition}

\begin{proof}
Iterate the branching law from Proposition \ref{branchsp2} using Theorems 2.1.2 and 2.2.2 from \cite{HTW}.  
\end{proof}

If $W^{k}_{\lambda'}$ appears in the above decomposition of $W^{nk}_{\lambda}$ for $n \ge 2$ with nonzero multiplicity, then $|\lambda'| \le |\lambda|$ and $|\lambda|$ and $|\lambda'|$ have the same parity.

Lastly, we will use the following result:

\begin{lemma}\label{lem:dag}
Let $\mf g = \mf{sl}, \mf{so}, \mf{sp}$.  Let $\boldsymbol\lambda$ be a dominant integral weight for $\mf g$.
  Then  $\Hom_{\mf g(n^k)} (\V{\lambda}{n^k}, \restr{\V{\lambda}{n^{k+1}}}{\mf g(n^k)}) \neq 0$ for $k\gg 0$.
\end{lemma}
\begin{proof}
Let $v \in \V{\lambda}{n^{k+1}}$ be a highest weight vector of weight $\boldsymbol{\lambda}$. Then $v$ remains a highest weight vector of weight $\boldsymbol{\lambda}$ when restricted to $\mf g(n^k)$.
\end{proof}


\section{Classification of finite type irreducible coherent local systems}\label{CLS}

Let $\mathfrak{g}_{\infty}$ be a locally simple, diagonal Lie algebra. Zhilinskii has classified the finite type coherent local systems for $\mathfrak{g}_{\infty}$ in \cite{Zh2}, which provides a description of ideals in $U(\mathfrak{g}_{\infty})$ of locally finite codimension. By \cite{Zh1}, this provides a complete description of the ideals of $U(\mathfrak{g}_{\infty})$ when $\mathfrak{g}_{\infty}$ is non-finitary.  In this section, we reproduce Zhilinskii's relevant results using our notation.

\begin{proposition}[\cite{Zh2}] \label{clsgl}
The irreducible coherent local systems of finite type for $\mathfrak{gl}(n^{\infty})$ are the following:
\begin{enumerate}
\item $\mathcal{A}^n_p$ for $p>0$, given by $C_k = \{V^{n^k}_{\lambda, (0)} \st |\lambda| = p\}$ for all but finitely many $k \in \mathbb{N}$
\item $\mathcal{B}^n_q$ for $q > 0$, given by $C_k = \{V^{n^k}_{(0), \mu} \st |\mu| = q\}$ for all but finitely many $k \in \mathbb{N}$
\item $\mathcal{C}^n_{p,q}$ for $p, q > 0$, given by $C_k = \{V^{n^k}_{\lambda, \mu} \st |\lambda| \le p, |\mu| \le q, |\lambda| - |\mu| = p - q \}$ for all but finitely many $k \in \mathbb{N}$
\end{enumerate}
\end{proposition}

We can relate representations of $\mf{gl}(n^\infty)$ to those of $\mf{sl}(n^\infty)$ by Lemma \ref{lem:foobar}. Therefore, the coherent local systems for $\mf{sl}(n^\infty)$ and $\mf{gl}(n^\infty)$ will be the same.

\begin{corollary}\label{cor:glslcls}
The irreducible c.l.s's for $\mf{sl}(n^\infty)$ are the same as those for $\mf{gl}(n^\infty)$ given in Proposition~\ref{clsgl}.
\end{corollary}
\begin{proof}
This  follows directly from Lemma~\ref{lem:foobar}.
\end{proof}

\begin{proposition}[\cite{Zh2}] \label{clssp}
The irreducible coherent local systems of finite type for $\mathfrak{sp}(n^{\infty})$ are:
\begin{itemize}
\item $\mathcal{D}^n_p$ for $p>0$, given by $C_k = \{U^{n^k}_{\lambda} \st |\lambda| \le p, |\lambda| \equiv p \mod 2 \}$ for all but finitely many $k \in \mathbb{N}$.
\end{itemize}
The irreducible coherent local systems of finite type for $\mathfrak{so}(n^{\infty})$ are:
\begin{itemize}
\item $\mathcal{E}^n_p$ for $p>0$, given by $C_k = \{W^{n^k}_{\lambda} \st |\lambda| \le p, |\lambda| \equiv p \mod 2 \}$ for all but finitely many $k \in \mathbb{N}$.
\end{itemize}
\end{proposition}


\section{The quiver of a coherent local system}\label{QUIVER}

In this short section, we show that a c.l.s. for $\mf{sl}(n^\infty)$, $\mf{so}(n^\infty)$, or $\mf{sp}(n^\infty)$ can be encoded in a quiver.  
To do this, we first note that the  branching laws from Section~\ref{BRANCH} give:

\begin{proposition}\label{prop:A}
Let $\mf g$ be one of $\mathfrak{sl}$, $\mathfrak{so}$, or $\mathfrak{sp}$.  
Let $C$ be an irreducible finite type c.l.s. for $\mathfrak{g}(n^\infty)$.
For all $k \gg 0$ there is a natural bijection $\alpha_k:  C_k \to C_{k+1}$, and this bijection can be chosen to be functorial in the sense that
\[ \Hom_{\mf g_k}(\V{\lambda}{n^k}, \restr{\V{\mu}{n^{k+1}}}{\mf g_k}) \cong \Hom_{\mf g_{k+1}}(\alpha_k(\V{\lambda}{n^k}), \restr{\alpha_{k+1}(\V{\mu}{n^{k+1}})}{\mf g_{k+1}})\]
for all $\V{\lambda}{n^k} \in C_k$ and $\V{\mu}{n^{k+1}} \in C_{k+1}$.
\end{proposition}

\begin{proof}
We provide a proof for $\mf g = \mf {sl}$ and note that the proof for $\mf g = \mf {sp}$ and $\mf g = \mf{ so}$ is analogous. 
It is enough to prove for a c.l.s. of type $\mc{C}_{p,q}$:  so let $C$ be the c.l.s. with $C_k = \{V^{n^k}_{\lambda, \mu} \st \abs{\lambda}\leq p, \abs{\mu} \leq q\}$ for all  $k \gg 0$.
We let $\alpha_k: C_k \to C_{k+1}$ be the obvious map, so $\alpha_k(V^{n^k}_{\lambda, \mu}) = V^{n^{k+1}}_{\lambda, \mu}$. Then $\alpha_k$ is clearly bijective. Note that the branching laws in section \ref{BRANCH} (Proposition \ref{branchgln}, \ref{branchspn}, and \ref{branchson}) depend on the choice of $n$ but do not depend on the choice of $k$. From this it follows that $\Hom_{\mf g_k}(V^{n^k}_{\lambda, \mu}, \restr{V^{n^{k+1}}_{\lambda', \mu'}}{\mf g_k}) \cong \Hom_{\mf g_{k+1}}(V^{n^{k+1}}_{\lambda, \mu}, \restr{V^{n^{k+2}}_{\lambda', \mu'}}{\mf g_{k+1}})$, which proves functoriality of $\alpha_k$. 
\end{proof}
For each irreducible finite type c.l.s. $C$ there is a natural set of partitions (or pairs of partitions) $\Lambda$ that is in bijection with $C_k$ for all $k \gg 0$. 
Going forward, we write $C = C(\Lambda)$, and for 
 the remainder of the section, we will  write $C_k = \{ \V{\lambda}{n^k} \}_{\lambda\in \Lambda}$. 

 Let $\mf g$ be one of $\mf{sl}$, $\mf{so}$, $\mf{sp}$, and let $C$ be a c.l.s. for $ \mf g$.  
We associate to $C$ a quiver $Q = Q(C)$, which we construct as follows.
The vertex set $Q_0$ is $ \Lambda$, and the number of arrows between vertices $\lambda$ and $\mu$ is $\dim_{\FF} \Hom_{\mf g_k}(\V{\lambda}{n^k}, \restr{\V{\mu}{n^{k+1}}}{\mf g_k})$. 
For  example, for $\mf{sl}(2^\infty)$ four of the quivers that occur are:
\beq\label{quivers}
\begin{array}{ccccccc}
\begin{tikzpicture}[bend angle=20, looseness=1]
\node (a) at (-1,0) {$(1)$};
 \draw[->]  (a) edge [in=55,out=120,loop,looseness=4] node[gap] {$\scriptstyle 2$}  (a);
\end{tikzpicture} & &
\begin{tikzpicture}[bend angle=20, looseness=1]
\node (a) at (-1,0) {$(2)$};
\node (b) at (1,0)  {$(1,1)$};
\draw[->,bend left] (b) to  (a);
\draw[->,bend left] (a) to  (b);
 \draw[->]  (a) edge [in=55,out=120,loop,looseness=4] node[gap] {$\scriptstyle 3$}  (a);
 \draw[->]  (b) edge [in=55,out=120,loop,looseness=4] node[gap] {$\scriptstyle 3$}  (b);
\end{tikzpicture} & &
 \begin{tikzpicture}[bend angle=20, looseness=1]
\node (a) at (-1,0) {$(3)$};
\node (b) at (1,0)  {$(2,1)$};
\node (c) at (3,0)  {$(1,1,1)$};
\draw[->,bend left] (b) to node[gap] {$\scriptstyle 2$} (a);
\draw[->,bend left] (a) to node[gap] {$\scriptstyle 2$} (b);
\draw[->,bend left] (b) to node[gap] {$\scriptstyle 2$} (c);
\draw[->,bend left] (c) to node[gap] {$\scriptstyle 2$} (b);
 \draw[->]  (a) edge [in=55,out=120,loop,looseness=4] node[gap] {$\scriptstyle 4$}  (a);
 \draw[->]  (b) edge [in=55,out=120,loop,looseness=4] node[gap] {$\scriptstyle 6$}  (b);
 \draw[->]  (c) edge [in=55,out=120,loop,looseness=4] node[gap] {$\scriptstyle 4$}  (c);
\end{tikzpicture}
&&
\begin{tikzpicture}[bend angle=20, looseness=1]
\node (a) at (-1,0) {$((1),(1))$};
\node (b) at (1,0)  {$((),())$};
\draw[->,bend right] (b) to node[gap] {$\scriptstyle 3$} (a);
 \draw[->]  (a) edge [in=55,out=120,loop,looseness=4] node[gap] {$\scriptstyle 4$}  (a);
 \draw[->]  (b) edge [in=55,out=120,loop,looseness=4]  (b);
\end{tikzpicture}
.
\end{array}
\eeq

The quivers are a way to encode the branching laws, but they also give another way to analyse locally simple representations of $\mf g_\infty = \mf g(n^\infty)$.  
Let $V = \lim V_k$ be such a representation, and let $I := \Ann_{U(\mf g_\infty)}(V)$. 
Let $Q := Q(C(I))$ be the quiver associated to the c.l.s. $C(I)$.  
In Section~\ref{QGR}, we will show that the category of left modules over $U(\mf g_\infty)/I$ is equivalent to a category of representations of $Q$.

 \part{Noncommutative geometry and category equivalences}\label{PART2}
 
 \section{Ultramatricial algebras and Bratteli diagrams}\label{ULTRA}

An {\em ultramatricial} algebra is a direct limit
\[ R = \varinjlim ( \xymatrix{ R_1 \ar[r]^{\phi_1} & R_2 \ar[r]^{\phi_2} & R_3 \ar[r]^{\phi_3} & \cdots}),\]
where the $R_i$ are finite dimensional semisimple algebras and the $\phi_i$ are (unital) algebra homomorphisms.  
We will see that ultramatricial algebras play a crucial role in our analysis of the representations of $U(\mf g(n^\infty))$.
We collect the results we will need on ultramatricial  algebras here.  

Let $R = \varinjlim_{\phi_i} R_i$ be an ultramatricial algebra.  
To $R$, we can associate a combinatorial object, its {\em Bratteli diagram}.  
The Bratteli diagram is an infinite graph which determines the $K$-theory and the Morita equivalence class of $R$.

To construct the Bratteli diagram of $R$, write each $R_i$ as $R_i^1 \oplus \dots \oplus R_i^{p_i}$, where the $R_i^j$ are simple.
Since the base field $\FF$ is algebraically closed, each $R_i^j$ is a matrix ring over $\FF$.
For each $i, j$, let $M_i^j$ be the natural $R_i^j$-module.  
For each $k \in \{1, \dots, p_{i+1}\}$, write 
\[\restr{M^k_{i+1}}{R_i^j} \cong \bigoplus_{j=1}^{p_i} (M_i^j)^{a_i^{jk}}.\]

The non-negative integers $a^{jk}$ are recorded in the {\em Bratteli diagram} for $R$.  This is a graph with a vertex $v_i^j$ for each $M_i^j$, and an edge of  multiplicity $a_i^{jk}$ between $v_i^j$ and $v_{i+1}^k$.
We typically write a Bratteli diagram as a horizontal strip with the vertices $v_i^j$ for fixed $i$ drawn one above the other, and a multiplicity label on each edge.

The  $a_i^{jk}$ also appear in the maps $\phi_i$.  
Let $k \in \{ 1, \dots, p_{i+1}\}$.  
Then $\phi_i$ determines a map from $R_i \to R_{i+1}^k$ which sends
\[ (A_1, \dots, A_{p_i}) \mapsto \operatorname{diag}(\underbrace{A_1, \dots, A_1}_{a_i^{1k}}, \underbrace{A_2, \dots,  A_2}_{a_i^{2k}}, \dots, 
\underbrace{A_{p_i}, \dots, A_{p_i}}_{a_i^{p_i k}}).\]
Note that given a Bratteli diagram for $R$, together with the degrees of the $R_i^j$ for some $i$, we may recover $R$ up to isomorphism.

Most Bratteli diagrams we will see will all look the same after a certain point.  More formally, suppose that for all $i \geq i_0$ we have $p_i = p_{i+1} = \dots$, and that $a_i^{jk} = a_{i+1}^{jk} = \dots$ for all $j, k \in \{1, \dots, p_{i_0}\}$.  
Such a diagram is called {\em stationary}.  
In a stationary Bratteli diagram we define $p= p_{i_0}$ and $a^{jk} = a_{i_0}^{jk}$.
If $R$ is  an ultramatricial algebra with a stationary Bratteli diagram, we say that $R$ itself is {\em stationary}. 

A stationary Bratteli diagram can be ``rolled up'' to give a quiver $Q$.  
This quiver has $p$ vertices, and $a^{jk}$ edges from vertex $j$ to vertex $k$.
If $R$ is a stationary ultramatricial algebra whose Bratteli diagram gives the quiver $Q$, we say that $Q$ is the {\em quiver of $R$}.

\begin{example}\label{exone}
Suppose that $R_i = M_{2^i}(\FF)$ and that the map $R_i \to R_{i+1}$ is $A \mapsto \begin{bmatrix} A & \\ & A\end{bmatrix}$.
The associated Bratteli diagram is
\[ \begin{array}{c}
\begin{tikzpicture}
\node(a) at (-1,0) {$\bullet$};
\node(b) at (0,0) {$\bullet$};
\node(c) at (1,0) {$\bullet$};
\node(dots) at  (1.5,0) {$\dots$};
\draw[-] (a) to  node[gap] {$\scriptstyle 2$} (b);
\draw[-] (b) to  node[gap] {$\scriptstyle 2$} (c);
\end{tikzpicture}
\end{array}
\]
and the quiver of $R = \varinjlim R_i$ is
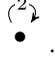
\[ 
\begin{array}{c}
\begin{tikzpicture}[bend angle=20, looseness=1]
\node (a) at (-1,0) {$\bullet$};
 \draw[->]  (a) edge [in=55,out=120,loop,looseness=4] node[gap] {$\scriptstyle 2$}  (a);
\end{tikzpicture} .
\end{array}
\]
\end{example}

\begin{example}\label{extwo}
The Bratteli diagram
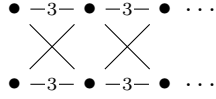
\[
\begin{array}{c}
\begin{tikzpicture}
\node(a1) at (-1,0) {$\bullet$};
\node(a2) at (-1,1) {$\bullet$};
\node(b1) at (0,0) {$\bullet$};
\node(b2) at (0,1) {$\bullet$};
\node(c1) at (1,0) {$\bullet$};
\node(c2) at (1,1) {$\bullet$};
\node(d1) at (1.5,0) {$\dots$};
\node(d2) at (1.5,1) {$\dots$};
\draw[-] (a1) to  node[gap] {$\scriptstyle 3$} (b1);
\draw[-] (a2) to  node[gap] {$\scriptstyle 3$} (b2);
\draw[-] (a1) to   (b2);
\draw[-] (a2) to   (b1);
\draw[-] (b1) to  node[gap] {$\scriptstyle 3$} (c1);
\draw[-] (b2) to  node[gap] {$\scriptstyle 3$} (c2);
\draw[-] (b1) to   (c2);
\draw[-] (b2) to   (c1);
\end{tikzpicture} 
\end{array}
\]
is associated  to the quiver 
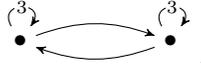
\[ 
\begin{array}{c}
\begin{tikzpicture}[bend angle=20, looseness=1]
\node (a) at (-1,0) {$\bullet$};
\node (b) at (1,0)  {$\bullet$};
\draw[->,bend left] (b) to  (a);
\draw[->,bend left] (a) to (b);
 \draw[->]  (a) edge [in=55,out=120,loop,looseness=4] node[gap] {$\scriptstyle 3$}  (a);
 \draw[->]  (b) edge [in=55,out=120,loop,looseness=4] node[gap] {$\scriptstyle 3$}  (b);
\end{tikzpicture} .
\end{array}
\]
Note that this quiver and the previous one appear in \eqref{quivers}.
\end{example}

Since the algebra $R$ depends both on the Bratteli diagram and on degree data for the $R_i^j$, the isomorphism class of a stationary ultramatricial algebra is not determined by the associated quiver.   However, we do have the following result of Goodearl:
\begin{proposition}\label{prop:Morita}
{\rm (\cite[Proposition~7.4]{S})}
Let $R, S$ be stationary ultramatricial $\FF$-algebras with isomorphic associated quivers.  Then $R$ is Morita equivalent to $S$.
\end{proposition}
We give a few details of the proof of Proposition~\ref{prop:Morita}.
If $R$ and $S$ are ultramatricial algebras, then by results of Elliot  \cite{Elliot}, if  $\phi:   (K_0(R), K_0(R)_+) \to (K_0(S), K_0(S)_+)$ is an isomorphism of ordered groups, then any $M \in S \lMod$ with $\phi([R]) = [M]$ gives a  Morita equivalence between $R$ and $S$ that induces $\phi$ at the level of $K$-theory.  
Note that there may be several isoclasses of modules $M$ that induce the same action on $K_0$.
The precise statement is that  an isomorphism $\phi$ determines a Morita equivalence between the two algebras only up to  {\em locally inner} automorphism \cite[Proposition~3.1]{BGR}.
(An automorphism $\alpha$ of an ultramatricial algebra $R = \varinjlim R_k$ is {\em locally inner} if any finite subset of $R$ is contained in some $R_k$ on which $\alpha$ restricts to an inner automorphism.)

To use Elliot's result,  we must understand $K_0$ of $R$ and $S$.   
Write $R = \varinjlim R_i$;
then  
\beq\label{kstar}
 (K_0(R), K_0(R)_+)  = \varinjlim (K_0(R_i), K_0(R_i)_+) 
 \eeq
as ordered groups.
For simplicity, we suppose that $R$ is stationary, with associated quiver $Q$.  
Then for $i \gg 0$, we have $(K_0(R_i), K_0(R_i)_+) \cong (\ZZ^p, \ZZ^p_+)$, and the map
$K_0(R_i) \to K_0(R_{i+1})$ induced by $\phi_i$ can be seen to be right  multiplication by the matrix $A= (a^{jk})$.
Note that $A$ is the adjacency matrix of $Q$.

For example, suppose that the associated quiver of $R$ is
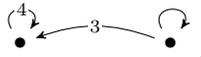
\[ 
\begin{array}{c}
\begin{tikzpicture}[bend angle=20, looseness=1]
\node (a) at (-1,0) {$\bullet$};
\node (b) at (1,0)  {$\bullet$};
\draw[->,bend right] (b) to node[gap] {$\scriptstyle 3$} (a);
 \draw[->]  (a) edge [in=55,out=120,loop,looseness=4] node[gap] {$\scriptstyle 4$}  (a);
 \draw[->]  (b) edge [in=55,out=120,loop,looseness=4]  (b);

\end{tikzpicture}
.
\end{array}
\]
The adjacency matrix is $\begin{bmatrix} 4 & 0 \\ 3 & 1 \end{bmatrix}$.
Thus the map $R_i \to R_{i+1}$ sends
\[ (M, N) \mapsto
 \Bigl( 
 \operatorname{diag}(M,M,M,M, N,N,N), N \Bigr).\]
We see that $R_{i+1} \otimes_{R_i} V_i^1 \cong (V_{i+1}^1)^4$ and
$R_{i+1} \otimes_{R_i} V_i^2 \cong (V_{i+1}^1)^{\oplus 3} \oplus V_{i+1}^2 $, and the map from $K_0(R_i) \to K_0(R_{i+1})$ is given by right multiplication by $A$. 

Now let $Q_R, Q_S$ be the quivers associated respectively to the stationary ultramatricial algebras $R = \varinjlim R_i$ and $S = \varinjlim S_i$.  
Suppose that $Q_R$ and $Q_S$ are isomorphic.
The isomorphism between the quivers determines a bijection of the vertices, and thus isomorphisms between $K_0(R_i)$ and $K_0(S_i)$ for all $i \gg 0$.
We further have that the diagrams 
\[ \xymatrix{ K_0(R_i) \ar[r]^{A} \ar[d] & K_0(R_{i+1}) \ar[d] \\
 K_0(S_i) \ar[r]_{A}  & K_0(S_{i+1})}
\]
commute, and there is therefore an isomorphism of ordered groups between $K_0(R)$ and $K_0(S)$. 
By Elliot's result, $R$ and $S$ are Morita equivalent.

We will also use Goodearl's characterization of simplicity of ultramatricial algebras via $K$-theory.  Recall that an {\em order-unit} in a partially ordered abelian group $(K, K_+)$ is an element $u \in K_+$ so that for any $x \in K$, there exists a positive integer $N$ with  $x \leq Nu$.
Goodearl's result is:
\begin{theorem}\label{thm:Gsimple}
{\rm ( \cite[Proposition~15.14]{Goodearl}) }
Let $R$ be  an ultramatricial algebra.
Then $R$ is a simple ring if and only if every nonzero element of $K_0(R)_+$ is an order-unit in $K_0(R)$.
\end{theorem}

We use this result to give a  sufficient condition for simplicity of a stationary ultramatricial algebra $R$. 
Recall that a quiver $Q$ is {\em strongly connected} if for each $i, j \in Q_0$, there is a path beginning at $i$ and ending at $j$.
We say that a strongly connected $Q$ is {\em primitive} if the greatest common divisor of the lengths of cycles in $Q$ is 1, although we note that the more usual terminology for this is that the adjacency matrix of $Q$ is primitive.  We say $Q$ is {\em symmetric} if the adjacency matrix of $Q$ is symmetric.  
Then we have:

\begin{proposition}\label{prop:simple}
Let $R = \varinjlim R_i$ be a stationary ultramatricial algebra, with associated quiver $Q$.
If $Q$ is strongly connected, primitive, and symmetric, then  $R$ is simple.
\end{proposition}
In particular,  the proposition shows that the algebras associated to the first three quivers of \eqref{quivers} are simple.

\begin{proof}
Let $K = K_0(R)$.  
We will apply Theorem~\ref{thm:Gsimple}; so we claim that any nonzero element of $K_+$ is an order-unit.  
To do this, we first characterise $K_+$, using the  fact \eqref{kstar} that $K_+ = \varinjlim K_0(R_i)_+$.
For any $u \in K_0(R_i)$ we denote the image of $u$ under the natural map to $K$ by $(u)$.

Let $A$ be the adjacency matrix of $Q$; it is $p\times p$, where $p = \# Q_0$.  We have seen that
$K = \varinjlim K_0(R_i )$ may be written as the injective limit
\[ K = \varinjlim ( \xymatrix{ K_0(R_i) \ar[r]^A & K_0(R_{i+1}) \ar[r]^A & \dots } ).\]
Here we identify $K_0(R_i)$ with $\ZZ^p$, and $A$ acts by right  multiplication on $\ZZ^p$, as above.
We make this identification in such a way that $K_0(R_i)_+$ is identified with $\ZZ^p_+$, and thus 
 $K_+ = \varinjlim K_0(R_i)_+ = \varinjlim_A ( \ZZ^p_+)$.

Let $y \in K$ with $y \neq 0$,  and write $y = (u)$ for some $u \in K_0(R_i)$.  
By the Perron-Frobenius theorem, the maximum modulus eigenvalue $\lambda_1$ of $A$ is in $\RR_+$ and has multiplicity $1$; further, it has an eigenvector $w_1 \in K_0(R) \otimes \RR$ all of whose entries are positive.
Since $A$ is symmetric, it is diagonalizable with real eigenvalues.
Perron-Frobenius thus implies that any eigenvector of any other eigenvalue of $A$ has a negative  entry.

We can if needed replace $A$ by $A^2$ to obtain without loss of generality  that all eigenvalues of $A$ are nonnegative.  
Let the other eigenvalues of $A$ be $\lambda_2, \dots, \lambda_p$, with $\lambda_1 > {\lambda_2} \geq \dots \geq {\lambda_p}$, 
and let $\{w_1, \dots, w_p\}$ be  a basis of eigenvectors for $\RR^p = K_0(R_i)\otimes \RR$, where the eigenvalue of $w_k$ is $\lambda_k$.
Then we have $u = \sum_k \alpha_k w_k$, for some $\alpha_k \in \RR$.  
We claim that $y \in K_+$ if and only if $\alpha_1 > 0 $.

To prove the claim, first suppose that $\alpha_1> 0$.  
Then for $j \gg 0$,  $A^j u$ is dominated by $\alpha_1 (\lambda_1)^j w_1$, and  since $A^j u \in K_0(R_{j+k}) =  \ZZ^p$ we can choose $j \gg 0$ so that $A^j u \in \ZZ^p_+$.
Thus $ y  = (\phi^j u) \in K_+$.

Similarly, if  $\alpha_1 < 0$ then we also see  no $A^j u$ will be positive, and thus $y \not\in K_+$.

If $\alpha_1 = 0$ then let $k$ be minimal with $\alpha_k \neq 0$ (thus $\lambda_k$ is maximal).  
Again, for $j \gg 0$,  the class $ A^j u \in \ZZ^p$ will be dominated by $(\lambda_i)^j w_i'$, where $w_i'$ is a nonzero vector in the  $\lambda_i$-eigenspace of $A$.  
Since some of the components of $w'_i$ are negative,
this will force some components of $A^j u$ to be negative. 
As $K_+ = \varinjlim K_0(R_i)_+$ we see that $y \not \in K_+$.

Now let $y = (u) \in K_+$.  We will show that $y$ is an order-unit.  Let $z = (v) \in K$.
(Here we have $u, v \in K_0(R_i)$ for some $i$.)
Let the $\lambda_1$-eigenspace components of $u, v$ be respectively $\alpha w_1$ and $\beta w_1$ for some $\alpha, \beta \in \RR$.
By the characterisation of $K_+$, we have $\alpha >0$.
Choose an integer $N > \beta/\alpha$, so $N \alpha - \beta > 0$.
Again by the characterisation of $K_+$, we have $Ny-z \in K_+$.
Thus $Ny > z$ and $y$ is an order-unit as claimed.

By Theorem~\ref{thm:Gsimple}, $R$ is simple.
\end{proof}

We will apply Proposition~\ref{prop:simple} later in the paper.

\section{Coherent local systems, quivers, and path algebras}\label{QGR}
In this section, we   apply the machinery of Section~\ref{ULTRA} to prove Theorem~\ref{ithm:catequiv}.  
To do so, we use results of S. Paul Smith \cite{S} on module categories associated to path algebras of quivers.
We will follow Smith's notation.
Let $Q$ be a finite quiver; the path algebra $\FF Q$ is naturally graded by path length.
Our convention  is that the path $\xymatrix{\lambda \ar[r]^{p} & \mu \ar[r]^{q} & \nu}$ is written $qp$, and that the vertex idempotents are written $e_\lambda$.
The set of arrows from $\lambda$ to $\mu$ is thus $e_\mu Q_1 e_\lambda$.

Denote the category of $\ZZ$-graded left $\FF Q$-modules by $\Gr{\FF Q}$; homomorphisms in $\Gr{\FF Q}$ preserve degree.
Let $\Fdim{\FF Q}$ be the full subcategory of $\Gr{\FF Q}$ consisting of locally finite-dimensional modules:  i.e. graded modules that are sums of their finite-dimensional submodules.
Then $\Fdim{\FF Q}$ is a localising subcategory of $\Gr{\FF Q}$, and we may form the quotient category
\[ \Qgr{\FF Q} := \Gr{\FF Q}/\Fdim{\FF Q}.\]

The construction of the category $\Qgr{\FF Q}$ is a standard technique in noncommutative algebraic geometry.  
For example, if $R = \FF[x_0, \dots, x_n]$ then it is a theorem of Serre that $\Qgr{R} \simeq \QCoh(\PP^n)$, the category of quasicoherent sheaves on $\PP^n = \PP^n_\FF$.
Because of this, the category $\Qgr{\FF Q}$ may be thought of as a ``category of sheaves on a noncommutative space,'' although we caution the reader that in the setting of $\FF Q$, this noncommutative space does not actually exist!

Let $\pi^*:  \Gr{\FF Q} \to \Qgr{\FF Q}$ be the natural quotient functor, and let $\sO := \pi^*( \FF Q)$.
If we think of $\Qgr{\FF Q}$ as sheaves on a nonexistent space, then $\sO$ plays the role of the structure sheaf.
One of the main results of \cite{S} is:
\begin{theorem}\label{thm:Smith}
 {\rm (\cite[Theorem~1.1]{S})}
Let $S(Q) := \End_{\Qgr{\FF Q}}(\sO)^{op}$.
Then $S(Q)$ is  an ultramatricial algebra  and $\Hom_{\Qgr{\FF Q}}(\sO, \blank)$ gives a category equivalence  from $\Qgr{\FF Q}$ to $ S(Q)\lMod$, the category of left $S(Q)$-modules.
Further, if $Q$ has no sources or sinks then  $S(Q) \lMod \simeq \Gr{L(Q)}$, where $L(Q)$ is the Leavitt path algebra of $Q$.
\end{theorem}
(For information about Leavitt path algebras, see \cite{Goodearl-Leavitt}.)

Smith also shows \cite[Section~5]{S} that the Bratteli diagram of $S(Q)$ is stationary with associated quiver $Q$.
We give a summary of the construction.  Let the vertex set of $Q$ be $\Lambda$. Smith shows that $S(Q) \cong \varinjlim S_k$, where
\beq\label{trex}
 S_k = \End_{\Gr {\FF Q}}(\FF Q_{\geq k})^{op} \cong \End_{\FF \Lambda}(\FF Q_k)^{op}.
\eeq
Clearly  $S_k$ is semisimple,  with  a matrix ring of size $\# e_\lambda Q_k = \# \{ \mbox{ paths of length $k$ ending at $\lambda$ } \}$ associated to each $\lambda \in \Lambda$.
That is,
\[ S_k \cong \prod_{\lambda\in\Lambda} M_{\# e_\lambda Q_k}(\FF).
\]
Smith shows that the natural map from $S_k$ to $S_{k+1}$ is given by the adjacency matrix of $Q$.

We use Smith's result to show:
\begin{theorem}\label{thm:quivercat}
Let $\mf g_\infty = \mf{sl}(n^\infty)$, $\mf{sp}(n^\infty)$, or $\mf{so}(n^\infty)$.  
 Let $I \subseteq U(\mf g_\infty)$ be the nonzero annihilator of a locally simple module and let $Q = Q(C(I))$ be the associated quiver, as in Section~\ref{QUIVER}.  
Then $U(\mf g_\infty)/I$ is Morita equivalent to $S(Q)$.
In consequence, there is a fully faithful functor from $\Qgr{\FF Q} \to U(\mf g_\infty) \lMod$, which gives an equivalence from $\Qgr{\FF Q}$ to the subcategory of representations of $\mf g_\infty$ annihilated by $I$.
This category is also equivalent to $\Gr{L(Q)}$, where $L(Q)$ is the Leavitt path algebra of $Q$.
\end{theorem}

\begin{proof}
 Let $C = C(I)$ and let $\Lambda$ be an index set for $C$, as above.
Fix $k_0$ so that $\# C_k = \# \Lambda $ for $k \geq k_0$.
For $k \geq k_0$ and $\lambda \in \Lambda$, let $I_k(\lambda) := \Ann_{\mf g_k}(V_\lambda^k) $. 
Let $I_k := \bigcap \{ I_k(\lambda) \st \lambda \in \Lambda\}$.
Then $U(\mf g_\infty)/I = \varinjlim U(\mf g_k)/I_k$.

Now, each $I_k(\lambda)$ is a maximal ideal of $U(\mf g_k)$.
It follows that
\beq \label{star}
U(\mf g_k)/I_k \cong \bigoplus_{\lambda \in \Lambda} U(\mf g_k)/I_k(\lambda).
\eeq
Each $U(\mf g_k)/I_k(\lambda)$ is a matrix ring of size $\dim_\FF V_\lambda^k$.

We need to understand the maps from  $U(\mf g_k)/I_k \to  U(\mf g_{k+1})/I_{k+1}$.
By \eqref{star}, it suffices to understand the maps $U(\mf g_k)/I_k(\lambda) \to  U(\mf g_{k+1})/I_{k+1}(\mu)$ for all $\lambda, \mu \in \Lambda$.
So let $x \in U(\mf g_k)$, and assume that $x$ annihilates all elements of $C_k$ except for $V_\lambda^k$.
The image of $x$ in $U(\mf g_{k+1})/I_{k+1}(\mu)$ is determined by the action of $x$ on $\restr{V^{k+1}_\mu}{\mf g_k}$.
Now we have
\[ \restr{V^{k+1}_\mu }{\mf g_k} \cong \bigoplus_{\nu \in \Lambda} (V_\nu^k)^{\# e_\mu Q_1 e_\nu}.\]
We see that $x$ acts trivially on the summands with $\nu \neq \lambda$, and diagonally on $(V_\lambda^k)^{\# e_\mu Q_1 e_\lambda}$.
If we view $U(\mf g_k)/I_k(\lambda)$ and $ U(\mf g_{k+1})/I_{k+1}(\mu)$ as matrix rings, the induced ring homomorphism sends
\[ x \mapsto \begin{bmatrix} x & & \\ & \ddots & \\ && x \end{bmatrix},\]
where there are $\# e_\mu Q_1 e_\lambda$ copies.
Note that the number of copies of $x$ depends only on $\lambda $ and $\mu$, and not on $k$.

The algebra $U(\mf g_\infty)/I = \varinjlim U(\mf g_k)/I_k$ is associated to a  Bratteli diagram given by ``unrolling'' $Q$, as described in Section~\ref{ULTRA}.
  By the comment after Theorem~\ref{thm:Smith}, the same Bratteli diagram is also associated to the algebra  $S(Q)$.
By Proposition~\ref{prop:Morita}, 
 $S(Q)$ and $U(\mf g_\infty)/I$ are Morita equivalent.
It follows from Theorem~\ref{thm:Smith} that
\[ \Qgr{\FF Q} \simeq  S(Q) \lMod \simeq U(\mf g_\infty)/I \lMod.\]

By Lemma~\ref{lem:dag}, there is a loop at each vertex of $Q$, so $Q$ has no sources or sinks.  
By Theorem~\ref{thm:Smith}, the categories above are also equivalent to $\Gr{L(Q)}$.
\end{proof}

All of Theorem~\ref{ithm:catequiv} except the final sentence follows from Theorem~\ref{thm:quivercat}.

To end the section, we give several remarks.

\begin{remark}\label{rem:gldim}
In the situation of Theorem~\ref{thm:quivercat}, $U(\mf{g}_\infty)/I$ is almost never semisimple.  
Fix  a nonzero ideal $I$ that annihilates a locally simple module, and define $Q = Q(C(I))$ and $S(Q)$ as in the theorem.
Note that $S(Q) = \varinjlim R_i$ is a stationary ultramatricial algebra.
As in \eqref{kstar}, we have $K_0(S(Q)) \cong \varinjlim K_0(R_i)$, where there exists $p$ so that each $K_0(R_i) \cong \ZZ^p$.  
The maps in the direct limit are given by right multiplication by the adjacency matrix of $Q$.  
For $S(Q)$ to be semisimple, we would have to have $K_0(S(Q)) \cong \ZZ^n$ for some $n$.
This only happens if the adjacency matrix of $Q$ is the identity.
\footnote{We thank Ken Goodearl for this argument.}

We will see in Part~\ref{PART3} that if $n \geq 2$, the adjacency matrix of $Q$ is not the identity for any nontrivial representation of $\mf{sl}(n^\infty)$, $\mf{sp}(n^\infty)$, or $\mf{so}(n^\infty)$.
\end{remark}

\begin{remark}\label{rem:two}
Let 
\beq\label{A} \xymatrix{ \Theta:  U(\mf g_\infty)/I \lMod \ar[r]^(0.6){\sim} & \Qgr{\FF Q}} \eeq
be the category equivalence of Theorem~\ref{thm:quivercat}.
The construction gives $\Theta$ as  a composition of equivalences
\beq \label{B}
\xymatrix{ \Theta:  U(\mf g_\infty)/I \lMod \ar[r]^(0.6){\Phi} & S(Q) \lMod \ar[r]^{\Psi} &\Qgr{\FF Q},} \eeq
and we briefly discuss $\Phi$ and $\Psi$.
First, $\Psi$ is canonical, the quasi-inverse of $\Hom_{\Qgr{\FF Q}}(\sO, \blank)$.
On the other hand, $\Phi$ is induced by identifying the quivers of  $U(\mf g_\infty)/I$ and $S(Q)$, as in Proposition~\ref{prop:Morita}, and thus inducing an isomorphism of ordered groups between $K_0(U(\mf g_\infty)/I)$ and $K_0(S(Q))$.
As discussed in Section~\ref{ULTRA}, this isomorphism determines a Morita equivalence between the two algebras only up to  locally inner automorphism. 
Thus choosing an equivalence $\Theta$ in \eqref{A} involves a choice of $\Phi$ in \eqref{B}. 
\end{remark}

\begin{remark}\label{rem:three}
We note that the conclusions of Theorem \ref{thm:quivercat} will hold whenever $\mathfrak{g}_{\infty}$ is a locally finite Lie algebra and $I \subseteq U(\mathfrak{g}_{\infty})$ is a locally primitive ideal whose associated coherent local system $C(I)$ is of finite type and stabilizes: i.e.  there is a bijection from $C_k$ to $C_{{k+1}}$, for all $k \gg 0$, which is functorial in the sense of Proposition \ref{prop:A}. When $\mathfrak{g}_{\infty}$ is diagonal and non-finitary, it follows from Section \ref{BACKGROUND} that any prime ideal $I \subseteq U(\mathfrak{g}_{\infty})$ will satisfy these conditions. When $\mathfrak{g}_{\infty}$ is finitary, not every prime ideal will have an associated c.l.s. of finite type. In particular, our results hold for the three simple finitary Lie algebras $\mathfrak{sl}(\infty)$, $\mathfrak{sp}(\infty)$, and $\mathfrak{o}(\infty)$ only when we consider ideals $I \subseteq U(\mathfrak{g}_{\infty})$ which satisfy the conditions above. 
\end{remark}

 \section{Locally simple modules,  point data sequences, and point representations of $Q$}\label{POINTS}
 
 In this section, let $\mf g_\infty = \mf{sl}(n^\infty)$, $\mf{sp}(n^\infty)$, or $\mf{so}(n^\infty)$.  
We write $\mf g_\infty = \varinjlim \mf g_k$, where $\mf g_k = \mf{sl}(n^k)$, $\mf{sp}(n^k)$, or $\mf{so}(n^k)$, respectively.

Let $X = \varinjlim_{x_k} X_k$ be a locally simple representation of $\mf g_\infty$;
we further assume that $\Ann_{U(\mf g_\infty)}(X) \neq 0$.
When writing $X$ in this way, we will assume that each $X_k$ is a (finite-dimensional) simple representation of $\mf g_k$ and that $x_k \in \Hom_{\mf g_k}(X_k, X_{k+1} )$. 
 We also assume that the $x_k$ are nonzero.
To $X$ is associated both discrete data, namely the isomorphism class of each $X_k$, and continuous data, the maps $x_k$.  
In this section, we explore how the isomorphism class of $X$ depends on this data, and prove a more general version of Proposition~\ref{iprop:pointseq}.

We first establish notation for the discrete invariant of $X$.  If  $I = \Ann_{U(\mf g_\infty)}(X)$ is the nonzero annihilator of a locally simple $X$, then $I$ is the ideal associated to a c.l.s. $C = C(\Lambda)$:  thus for $k \gg 0$, there is $\lambda_k \in \Lambda$ with $X_k \cong \V{\lambda_k}{k}$.  We refer to the sequence $(\lambda_\bullet)$ as the {\em $\Lambda$-sequence} of $X$.

We begin with two trivial but useful observations.

\begin{lemma}\label{isomorphic}
Suppose $V = \varinjlim_{v_k} V_k$ and $W = \varinjlim_{w_k} W_k$ are $\mf g_\infty$-modules. Then $V \cong W$ if and only if there exist subsequences $i_1 < i_2 < i_3 < \dots$ and $j_1 < j_2 < j_3 < \dots$, with $i_1 \leq j_1 \leq i_2 \leq j_2 \dots$, and  $\phi_\ell\in\Hom_{\mf g_{i_\ell}}( V_{i_\ell}, W_{j_\ell})$ and  $\phi_\ell' \in \Hom_{\mf g_{j_\ell}}(W_{j_\ell} , V_{i_{\ell+1}})$ for all $\ell$ such that the following diagram commutes:

\begin{eqnarray*}
\xymatrixcolsep{5pc}
\xymatrix{
V_{i_1} \ar[d]_{\phi_1}  \ar[r]^{f_1} &V_{i_2} \ar[d]_{\phi_2} \ar[r]^{f_2} &V_{i_3} \ar[d]_{\phi_3} \ar[r]^{f_3} &\dots\\
W_{j_i} \ar[r]_{g_1}  \ar[ru]^{\phi_1'} &W_{j_2} \ar[ru]^{\phi_2'}  \ar[r]_{g_2} &W_{j_3} \ar[r]_{g_3} \ar[ru]^{\phi_3'} &\dots .} 
\end{eqnarray*}
(Here the horizontal maps are induced from the $v_k$ and $w_k$.)
\qed
\end{lemma}

The $\Lambda$-sequence of a locally simple $U(\mf g_\infty)$-module $X$ is not an isomorphism invariant, but we have:
\begin{corollary}\label{samemodules}
Let $V \cong W$ be locally simple $\mf g_\infty$-modules, and let $(\lambda_\bullet)$, $(\mu_\bullet)$ be their respective $\Lambda$-sequences. 
Then $\lambda_k = \mu_k$ for $k \gg 0$.
\end{corollary}
\begin{proof}
%
Let $i_1\leq  j_1 $ be as in Lemma~\ref{isomorphic}, and let $\ell \geq j_1$.  We will show that $V_\ell \cong W_\ell$.  Let $n$ be such that $j_n \geq \ell$.  From the commutative diagram in Lemma~\ref{isomorphic}, there is a commutative diagram
\[
 \xymatrixcolsep{3pc}
\xymatrix{
V_{i_1} \ar[d]_{\phi_1}  \ar[rd]^{\phi} \ar[rrr]^{f} &&&V_{i_{n+1}} \\
W_{j_i}  \ar[r] & W_{\ell} \ar[rru]^{\phi'} \ar[r] &  W_{j_n} \ar[ru]_{\phi_n'} &},
\]
where the horizontal maps are induced from the $v_k$ and $w_k$, and all maps are homomorphisms over the appropriate $\mf{g}_k$.
By assumption, $f$ factors as the composition
\[ \xymatrix{V_{i_1} \ar[r]^{f'} \ar@/_/[rr]_{f} & V_{\ell} \ar[r]^{f''} & V_{i_{n+1}}}.\]
Since $f'' f' = f = \phi' \phi$, there are $\mf g_{\ell}$-maps
\[ V_{\ell} \cong f''(V_{\ell}) = f''(\mf g_{\ell} \cdot f'(V_{i_1})) = \mf g_{\ell} \cdot f''f'(V_{i_1}) = \mf g_{\ell} \cdot \phi'\phi(V_{i_1}) = \phi'(W_{\ell}) \cong W_{\ell},\]
as needed.  
\end{proof}

Let $X = \varinjlim_{x_k} X_k$ be a locally simple $\mf g(n^\infty)$-module, and for each $k$ let  $H_k = \Hom_{\mf g_k}(X_k, X_{k+1})$. 
By Schur's lemma, each $H_k$ is a finite-dimensional vector space.
Thus to $X$ we associate a sequence of points $( x_k \in H_k \ssm \{0\})$.
The sequence $(\lambda_\bullet, x_\bullet)$, where $(\lambda_\bullet)$ is the $\Lambda$-sequence of $X$, is the {\em point data sequence} of $X$.

We define an equivalence relation $\sim$ on point data sequences. 
 Given nonzero $x_k \in H_k$, let $[x_k]$ denote the associated point in $\PP(H_k)$.   
We say that 
\beq\label{twiddle}
\mbox{ $(\lambda_\bullet, x_\bullet) \sim (\mu_\bullet, y_\bullet)$ if  $\lambda_k = \mu_k$ and $[x_k] = [y_k]$  for all $k \gg 0$.}
\eeq
We next show this equivalence relation characterizes isomorphism of locally simple modules.
This result was proved independently by Cody Holdaway \cite{Ho}.

\begin{proposition} \label{sequences}
Let $X$ and $Y$ be  locally simple $\mf g_\infty$-modules with associated point data sequences $(\lambda_\bullet, x_\bullet)$ and $(\mu_\bullet, y_\bullet)$, respectively.
Then $X \cong Y$ $ \iff$  
$(\lambda_\bullet, x_\bullet) \sim (\mu_\bullet, y_\bullet)$.
\end{proposition}

\begin{proof}
By Corollary~\ref{samemodules}, we can assume without loss of generality that $X_k \cong Y_k$ (so $\lambda_k = \mu_k$) for all $k$.

$(\Longrightarrow)$.  By assumption, there are $a_k \in \FF^*$ so that $a_k x_k = y_k$ for each $k \geq k_0$.  
Consider the diagram:
\begin{eqnarray*}
\xymatrixcolsep{5pc}
\xymatrix{
X_{1} \ar[d]_{1}  \ar[r]^{x_1} &X_{2} \ar[d]_(0.6){a_1} \ar[r]^{x_2} &X_{3} \ar[d]_(0.6){a_1a_2} \ar[r]^{x_3} &X_4 \ar[r] \ar[d]_(0.6){a_1a_2a_3}&\dots\\
X_{1} \ar[r]_{y_1}  \ar[ru]^{x_1} &X_{2} \ar[ru]^{x_2/a_1}  \ar[r]_{y_2} &X_{3} \ar[ru]^{x_3/{a_1a_2}} \ar[r]_{y_3} &X_4 \ar[r] &\dots}
\end{eqnarray*}
\noindent where the maps $X_i \rightarrow X_i$ are given by multiplying by the scalar indicated.
This diagram trivially commutes, and by  Lemma~\ref{isomorphic} we have $X \cong Y$.

Before we prove the other direction,  consider the map 
\[\iota:  \PP(\Hom_{\mf g_i}(X_i, \restr{X_j}{\mf g_i})) \times \PP (\Hom_{\mf g_j}(X_j, \restr{X_k}{\mf g_j})) \to \PP (\Hom_{\mf g_i}(X_i, \restr{X_k}{\mf g_i}))\]
induced by composition of functions.
If the multiplicity of $X_j$ in $\restr{X_k}{\mf g_j} $ is $b$ and the multiplicity of $X_i$ in 
$\restr{X_j}{\mf g_i} $ is $a$, then the multiplicity  of $X_i$ in 
$\restr{X_k}{\mf g_i}  $ is at least $ab$.
Let $\phi = (\phi_1, \dots, \phi_a) \in \Hom_{\mf g_i}(X_i, \restr{X_j}{\mf g_i})$ and
$\psi = (\psi_1, \dots, \psi_b) \in \Hom_{\mf g_j}(X_j, \restr{X_k}{\mf g_j})$, where $\psi_s, \phi_r \in \FF$.
Then 
\[ \psi \circ \phi = (\psi_1 \phi_1, \dots, \psi_1 \phi_a, \psi_2 \phi_1, \dots, \psi_2 \phi_a, \dots, \psi_b\phi_1, \dots, \psi_b \phi_a),\]
in an appropriate choice of coordinates.
Thus $\iota$ may be given in coordinates as the Segr\'{e} embedding $\PP^{a-1} \times \PP^{b-1} \to \PP^{ab-1} \subseteq \PP (\Hom_{\mf g_i}(X_i, \restr{X_k}{\mf g_i}))$ and, in particular,  $\iota$ is injective.

$(\Longleftarrow)$.  Assume that $X \cong Y$.  
By Corollary~\ref{samemodules}, $X_k \cong Y_k$ for $k \gg 0$.
By Lemma~\ref{isomorphic} there is a commutative diagram
 \begin{eqnarray*}
\xymatrixcolsep{5pc}
\xymatrix{
X_{i_1} \ar[d]_{\phi_1}  \ar[r] &X_{i_2} \ar[d]_{\phi_2} \ar[r] \ar[d]_{\phi_3} \ar[r] &\dots\\
X_{j_i} \ar[r]  \ar[ru]^{\phi_1'} &X_{j_2} \ar[ru]^{\phi_2'}  \ar[r] \ar[r] \ar[ru]^{\phi_3'} &\dots},
\end{eqnarray*}
where the top maps are compositions of the relevant $x_k$'s, and the bottom maps are compositions of the relevant $y_k$'s.
Thus $\phi_1' \phi_1 = x_{i_2-1} \cdots  x_{i_1}$, and injectivity of $\iota$ implies that $\phi_1' = x_{i_2-1}\dots x_{j_1}$ up to scalar.  
Likewise, since $\phi_2 \phi_1' = y_{j_2-1} \cdots  y_{j_1}$, we have that $\phi_1' = y_{i_2-1}\dots y_{j_1}$ up to scalar.
Injectivity of $\iota$ and induction gives that
 the natural map
\[ \PP(H_{i_1}) \times \PP(H_{i_1+1}) \times \dots \times \PP(H_{i_2 -1}) \to \PP(\Hom_{\mf g_{i_1}}(X_{i_1}, \restr{X_{i_2}}{\mf g_{i_1}}))\]
is injective, and so  $[x_k] = [y_k]$ for $k \in \{ j_1 , \dots, i_2-1\}$.
We can repeat this argument to obtain that $[x_k] = [y_k]$ for all $k \geq j_1$.  
\end{proof}

Proposition~\ref{iprop:pointseq} is the special case of Proposition~\ref{sequences} where $\mf{g} = \mf{sl}$ and $\lambda_k = (1) $ for all $k$.

Our goal for the rest of the  section is  to obtain Proposition~\ref{sequences} from Theorem~\ref{thm:quivercat} and complete the proof of Theorem~\ref{ithm:catequiv}.  
We fix a c.l.s $C = C(\Lambda)$ for $\mf g_\infty$ and let $I = I(C)$ and $Q= Q(\Lambda)$.
By Theorem~\ref{thm:quivercat} the categories $U(\mf g_\infty)/I \lMod$ and $\Qgr{\FF Q}$ are equivalent, as in \eqref{A}.
As discussed 
in Remark~\ref{rem:two},
 there are many choices of the category equivalence $\Theta$ from \eqref{A}, 
and the image of a particular locally simple $U(\mf g_\infty)/I$-module of course depends on that choice.
However, regardless of the choice of $\Theta$, the image of a locally simple $U(\mf{g_\infty})/I$-module is a special kind of representation of $Q$.
\begin{defn} \label{def:point}
 A {\em point representation} of $Q$ is a cyclic graded $\FF Q$-module $M$ with $\dim_\FF M_k = 1$ for $k \gg 0$.

If $M$ is a point representation then for $k \gg 0 $ we have $M_k \cong \FF e_{\lambda_k}$ for some $\lambda_k \in \Lambda$.
We refer to $(\lambda_\bullet)$ as the {\em $\Lambda$-sequence of $M$}.

We also abuse notation and refer to the equivalence class $\pi^* M$ in $\Qgr{\FF Q}$ as a point representation of $Q$.
\end{defn}

We first discuss the equivalence $\Psi$ from \eqref{B}.
Define $S_k = \End_{\FF \Lambda} (\FF Q_k)^{op}$ as in \eqref{trex}, so $\FF Q_k$ gives a Morita equivalence between $S_k$ and $\FF \Lambda$.
For each $k \in \NN$ there is a natural identification of $\Lambda$ with the set of simple left $S_k$-modules:  let
\[ U_\lambda^k = \Hom_{\FF \Lambda}(\FF Q_k, \FF e_\lambda).\]
Since there is a loop at each vertex of $Q$ by Lemma~\ref{lem:dag},
each $U_\lambda^k$ is nonzero and we further have
\[ \FF e_\lambda \cong \FF Q_k \otimes_{S_k} U_\lambda^k\]
for all $\lambda, k$.

\begin{lemma}\label{lem:theta}
 For all $\lambda, \mu \in \Lambda$ there is an isomorphism 
\[ \theta_k:  \Hom_{\FF \Lambda}(\FF Q_1 \otimes_{\FF \Lambda} \FF e_\lambda, \FF e_\mu) \to\Hom_{S_k}(U_\lambda^k, U_\mu^{k+1})\]
given by
\[
\theta_k(t)(\alpha)(ap) = t(a \otimes \alpha(p))
\]
for all $t \in \Hom_{\FF \Lambda}(\FF Q_1 \otimes_{\FF \Lambda} \FF e_\lambda, \FF e_\mu)$,  $\alpha \in U_\lambda^k$, $a \in Q_1$, $p \in Q_k$.
\end{lemma}
\begin{proof}
 We have
\begin{align*}
 \Hom_{\FF \Lambda}(\FF Q_1 \otimes_{\FF \Lambda} \FF e_\lambda,\FF e_\mu) 
& \cong \Hom_{\FF \Lambda} (\FF Q_1 \otimes_{\FF \Lambda} \FF Q_k \otimes_{S_k} U^k_\lambda,\FF Q_{k+1} \otimes_{S_{k+1}} U^{k+1}_\mu) 
\\
& \cong \Hom_{S_{k+1}}(S_{k+1} \otimes_{S_k} U^k_\lambda, U^{k+1}_\mu) 
\\
& \cong \Hom_{S_k}(U^k_\lambda, U^{k+1}_\mu).
\end{align*}
Thus both have the same (finite) dimension.

We claim that $ \theta_k$ is injective.  
So let $t \in \Hom_{\FF \Lambda}(\FF Q_1 \otimes_{\FF \Lambda} \FF e_\lambda, \FF e_\mu)$ be nonzero, and choose $a \in Q_1$ so that $t(a \otimes e_\lambda) \neq 0$.
Then for any $\alpha \in U^k_\lambda$ and $p \in Q_k$ we have that $t(a \otimes \alpha(p)) = 0$ if and only if $\alpha(p) = 0$, since $a \otimes \alpha(p)$ is a scalar multiple of $a \otimes e_\lambda$.
Thus $\theta_k(t)(\alpha) = 0$ if and only if $\alpha =0$ and $\theta_k(t) $ is injective.
As $U^k_\lambda \neq 0$, this says that $\theta_k(t)\neq 0$, as we needed.

By dimension count $\theta_k$ is an isomorphism.
\end{proof}

It is straightforward that $\Hom_{\FF \Lambda}(\FF Q_1 \otimes_{\FF \Lambda} \FF e_\lambda, \FF e_\mu)$ may be identified with $(e_\mu \FF Q_1 e_\lambda)^*$, using $e_\lambda$ and $e_\mu$ as basis elements, and we will sometimes refer to the pairing
\[ \ang{\ ,\ }:  \Hom_{S_k} (U^k_\lambda, U^{k+1}_\mu) \times e_\mu \FF Q_1 e_\lambda \to \FF \]
induced by $\theta_k$.

We use the maps $\theta_k$ to define a point module associated to any locally simple $S(Q)$-module.
Let $U = \varinjlim_{u_k} U^k_{\lambda_k}$ be such a module, where we take $k \geq k_0$.
We define a graded $\FF Q$-representation $M(\lambda_\bullet, u_\bullet)$ by:
\[ M (\lambda_\bullet, u_\bullet)= \bigoplus_{k \geq k_0} M_k, \quad \mbox{where $M_k = \FF e_{\lambda_k}$}.\]
The $\FF Q$-action is given by:
\[ a \cdot e_{\lambda_k} = \theta^{-1}_k(u_k)(a \otimes e_{\lambda_k})\]
for any $a \in Q_1$.  In terms of the pairing above, we write
\beq \label{boff} 
a \cdot e_{\lambda_k} = \ang{u_k, a}  e_{\lambda_{k+1}}.
\eeq

We next show that this construction gives us the functor $\Psi$.
\begin{proposition}\label{prop:impsi}
 Let $U = \varinjlim_{u_k} U^k_{\lambda_k}$ be a locally simple $S(Q)$-module.
Then $\Psi(U) \cong \pi^* M(\lambda_\bullet, u_\bullet)$.
\end{proposition}
 \begin{proof}
  The equivalence $\Psi$ was defined in terms of its quasi-inverse
\beq\label{defpsiinv} \Psi^{-1} = \Hom_{\Qgr{\FF Q}}(\sO, \blank)
 = \varinjlim_k \Hom_{\Gr{\FF Q}}(\FF Q_{\geq k}, \blank) \cong \varinjlim_k \Hom_{\FF \Lambda}(\FF Q_k, (\blank)_k).
\eeq

Let $N$ be a point representation of $\FF Q$ with associated $\Lambda$-sequence $(\lambda_\bullet)$.
We write the $Q$-action on $N$ as maps
$t:  \FF Q_{\ell} \otimes N_k \to N_{k+\ell}$.
Note also that  we can identify $\Hom_{\FF \Lambda}(\FF Q_k, N_k)$ with $U^k_{\lambda_k}$.

To compute $\Psi^{-1}(N)$, we must understand the induced maps from $U^k_{\lambda_k} \to U^{k+1}_{\lambda_{k+1}}$ in \eqref{defpsiinv}.
Following the isomorphisms, we see each such map is given as the composition
\[ \xymatrix{
 U_k \ar[r]^(0.25){\eta} & \Hom_{\Gr{\FF Q}}(\FF Q_{\geq k}, N_{\geq k}) \ar[r]^{\textsf{restr}} &  \Hom_{\FF \Lambda}(\FF Q_{k+1}, N_{ k+1}),}
\]
where here $\eta(\alpha)$ sends a path $qp$ (where $p \in Q_k$) to $t(q \otimes \alpha(p))$.
That is, the induced map from $U^k_{\lambda_k} \to U^{k+1}_{\lambda_{k+1}}$ is given by applying $\theta_k$ to the $Q$-action.
From \eqref{boff},  $\Psi^{-1}(\pi^* M(\lambda_\bullet, u_\bullet)) \cong \varinjlim_{u_k} U^k_{\lambda_k} = U$ as needed.
 \end{proof}

Proposition~\ref{prop:impsi} shows that $\Psi$ is defined canonically; in contrast, we have seen that $\Phi$ depends on a choice.
Fix an equivalence
\[ \xymatrix{\Phi:  U(\mf g_\infty)/I \lMod \ar[r]^(0.6){\sim} &  S(Q) \lMod}\]
as in \eqref{B}.
Recall that $U(\mf g_\infty)/I = \varinjlim U(\mf g_k)/I_k$.  
For any $\Lambda$-sequence $(\lambda_\bullet)$, the equivalence $\Phi$ induces isomorphisms
\beq \label{dreich}
\xymatrix{ \Hom_{U(\mf g_k)/I_k}(\V{\lambda_k}{k}, \V{\lambda_{k+1}}{k+1}) \ar[r]^(0.55){\sim} & \Hom_{S_k}(U^k_{\lambda_k}, U^{k+1}_{\lambda_{k+1}})}
\eeq
for $k \gg 0$.
Let us denote the isomorphisms in \eqref{dreich} by $\phi$ (or by $\phi_k$ if we want to keep track of the subscripts).
Thus
\beq \label{dreichdreich}
\Phi(\varinjlim_{v_k} \V{\lambda_k}{k}) = \varinjlim_{\phi_k(v_k)} U^k_{\lambda_k}.
\eeq

Given this notation, we have:
\begin{corollary} \label{cor:impsi}
 Choose $\Phi$ as in \eqref{B} and define $\Theta$ as in \eqref{A}.
Then
\[
 \Theta(\varinjlim_{v_k} \V{\lambda_k}{k}) \cong \pi^* M(\lambda_\bullet, \phi(v_\bullet))
\]
and in particular is a point representation.
\end{corollary}
\begin{proof}
 Combine \eqref{dreichdreich} and Proposition~\ref{prop:impsi}.
\end{proof}
Crollary~\ref{cor:impsi} completes the proof of Theorem~\ref{ithm:catequiv}.

We now give the promised alternate proof of Proposition~\ref{sequences}.
In fact, we prove:
\begin{theorem}\label{thm:123}
 Choose $\Phi$ as in \eqref{B} and define $\Theta$ as in \eqref{A}.
Let $V = \varinjlim_{v_k} V_k$ and $W = \varinjlim_{w_k}W_k$ be locally simple $U(\mf g_\infty)/I$-modules, with associated $\Lambda$-sequences $(\lambda_\bullet)$ and $(\mu_\bullet)$, respectively.
The following are equivalent:
\begin{enumerate} \item $V \cong W$
\item  $\pi^* M(\lambda_\bullet, \phi(v_\bullet)) \cong \pi^* M(\mu_\bullet, \phi(w_\bullet))$ in $\Qgr{\FF Q}$.
 \item $(\lambda_\bullet,v_\bullet) \sim (\mu_\bullet, w_\bullet)$ in the sense of \eqref{twiddle}.
\end{enumerate}
\end{theorem}
\begin{proof}
 By Corollary~\ref{cor:impsi} we have
$\Theta(V) \cong \pi^* M(\lambda_\bullet, \phi(v_\bullet))$ and $ \Theta(W) \cong \pi^* M(\mu_\bullet, \phi(w_\bullet))$.
Thus $(1)$ $\iff$ $(2)$ follows from Theorem~\ref{thm:quivercat}.

For $(2)$ $\iff$ $(3)$, let $x_\bullet = \phi(v_\bullet)$ and $y_\bullet = \phi(w_\bullet)$.
If $\lambda_k \neq \mu_k$ for arbitrarily large $k$, then clearly $\pi^* M(\lambda_\bullet, x_\bullet) \not\cong \pi^* M(\mu_\bullet, y_\bullet)$.
So assume that $\lambda_k = \mu_k$ for $k \geq k_0$ and let $\{m_k = e_{\lambda_k}\}$ be the fixed basis for $M(\lambda_\bullet, x_\bullet)$.
By \eqref{boff}, 
$\pi^* M(\lambda_\bullet, x_\bullet) \cong \pi^* M(\lambda_\bullet, y_\bullet)$ if and only if there are $m'_k \in M_k$ with $a \cdot m_k' = \ang{y_k, a} m_{k+1}'$ for all $k \gg 0$ and for all $a \in e_{\lambda_{k+1}} Q_1 e_{\lambda_k}$.
Since the $M_k$ are all one-dimensional, this happens if and only if $(x_\bullet) \sim (y_\bullet)$ in the sense of \eqref{twiddle}.
\end{proof}

\begin{remark}\label{rem:geomwild}
We note that for the $n$-loop one-vertex quiver
\[ 
\begin{array}{c}
\begin{tikzpicture}[bend angle=20, looseness=1]
\node (a) at (-1,0) {$\bullet$};
 \draw[->]  (a) edge [in=55,out=120,loop,looseness=4] node[gap] {$\scriptstyle n$}  (a);
\end{tikzpicture}
\end{array}
\]
 the category of point representations in $\Qgr{\FF Q}$ is {\em wild} in the sense of having subsets parameterised by varieties of arbitrarily large dimension.   In this setting, by Theorem~\ref{thm:123} we know that point representations of $\FF(Q)$ are parameterised by sequences of points in $\PP^{n-1}$, up to the equivalence relation $\sim$ defined in \eqref{twiddle}.  Let $n \geq 2$, and fix $d \in \ZZ_{\geq 1}$.  For any $x  = (x_1, \dots, x_d) \in (\PP^{n-1})^{\times d}$, define  a point representation $M(x)$ corresponding to the point sequence
\[ \underline{x}  = (x_1, \dots, x_d, x_1, \dots, x_d, x_1, \dots).\]
It is clear that if $x \neq y$ then $\underline{x} \not \sim \underline{y}$ and so the $M(x)$ form a set of (distinct) isomorphism classes of point representations of $\FF(Q)$ that is parameterised by $(\PP^{n-1})^{\times d}$.
It follows that the category of natural representations of $\mf{sl}(n^\infty)$ is wild in this geometric sense.

A similar argument, with more notational complexity, shows that $\Qgr{\FF Q}$ is wild in the sense above if $Q$ is any quiver with at least two oriented cycles based at a single vertex.  It follows from Theorem~\ref{thm:123} and the results in Part~\ref{PART3} that the category of locally finite representations of $\mf{sp}(n^\infty)$ or of $\mf{so}(n^\infty)$ is wild for any $n \geq 2$.
\end{remark}

\part{Combinatorics of quivers}\label{PART3}

Below we compute the adjacency matrices associated to each irreducible coherent local system. Fix a positive integer $n$. For $\mf sl(n^{\infty})$, we say that the the quivers and their adjacency matrices which are associated to the coherent local systems $\mathcal{A}^n_p$ and $\mathcal{B}^n_q$ are of ``Type I.''  The quivers and the adjacency matrices associated to $\mathcal{C}^n_{p,q}$ are of ``Type II''. 

\section{Type I quivers and characters of the symmetric group $S_d$}\label{type1}

In order to compute the adjacency matrices, and hence the quivers, for the Type I case, we only need to do the calculations for $\mathcal{A}$ since the matrices for $\mathcal{B}$ will be the same. We use  the  branching law for $\mathfrak{gl}(k) \hookrightarrow \mathfrak{gl}(nk)$ (which follows from Proposition \ref{branchgln}) to define coefficients $a^n_{\lambda, \mu}$:
\begin{eqnarray}\label{adef}
\restr{V^{nk}_{\lambda, 0}}{\mathfrak{gl}(k)} \cong \bigoplus_{\beta_1, \beta_2, \dots, \beta_n, \mu}  c^{\lambda}_{\beta_1, \beta_2, \dots, \beta_n}c^{\mu}_{\beta_1, \beta_2, \dots, \beta_n} V^k_{\mu, 0} = \bigoplus_{\mu} a^n_{\lambda, \mu} V^k_{\mu, 0} .
\end{eqnarray}
In particular, if $V^k_{\mu, 0}$ appears in the decomposition for $V^{nk}_{\lambda, 0}$, then $|\lambda| = |\mu|$. 

We define a matrix $A^n_d$ for the coherent local system $\mathcal{A}^n_{d}$ in the following way: For each $d$, we order the partitions of $d$ in reverse lexicographical order (as in \cite{McD}) so that the partition $(d)$ comes first and the partition $(1,1, \dots, 1)$ comes last. The entries  $A^n_d$ are given by $(a^n_{\lambda, \mu})$ where $\lambda$ and $\mu$ run over all partitions of $d$ in the order given. 

For example, suppose that $n = 2$. Then the entries of the matrices $A^2_d$ are given by the branching law \eqref{branchgl20}:
\begin{eqnarray*}
\restr{V^{2k}_{\lambda, 0}}{\mathfrak{gl}(k)} \cong \bigoplus_{\beta_1, \beta_2, \mu}  c^{\lambda}_{\beta_1, \beta_2}c^{\mu}_{\beta_1, \beta_2} V^k_{\mu, 0} . 
\end{eqnarray*}
It is easy to see that $A^2_1 = (2)$. We can also calculate that $A^2_2 = \left( \begin{array}{cc} 3 & 1 \\ 1 & 3 \end{array} \right)$, $A^2_3 = \left( \begin{array}{ccc} 4 & 2 & 0\\ 2 & 6 & 2 \\ 0 & 2 & 4 \end{array} \right)$, and 
\begin{eqnarray*}
A^2_4 = \left( \begin{array}{ccccc} 5 & 3 & 1 & 0 & 0 \\ 3 & 9 & 3 & 4 & 0 \\ 1 & 3 & 6 & 3 & 1 \\ 0 & 4 & 3 & 9 & 3 \\ 0 & 0 & 1 & 3 & 5 \end{array} \right) .
\end{eqnarray*}

From the branching law, we already have some basic information about $Q$.  In particular, $Q$ satisfies the hypotheses of Proposition~\ref{prop:simple}:

\begin{proposition}\label{prop:simple2}
Let $n \in \ZZ_{\geq 1}$ and $d \in \ZZ_{\geq 0}$.
Let $Q$ be a quiver of type $\mc A^n_d$ or $\mc B^n_d$.
Then $Q$ is strongly connected, primitive, and symmetric.  
As a result, if $C$ is a Type I c.l.s. for $\mf{sl}(n^\infty)$, then $U(\mf{sl}(n^\infty))/I(C) $ is simple.
\end{proposition}
\begin{proof}
It suffices to give a proof for $\mc A^n_d$.
The entries $a^n_{\lambda\mu}$ of the  adjacency matrix $A^n_d$ of $Q$ are non-negative by \eqref{adef}
The matrix $A^n_d$ is symmetric, since $a^n_{\lambda, \mu} = a^n_{\mu, \lambda}$ because these entries are defined in terms of Littlewood-Richardson coefficients.
That $Q$ is strongly connected follows from the irreducibility of the c.l.s. $\mc A^n_d$, and $Q$ is primitive since by Lemma~\ref{lem:dag} $Q$ contains loops.  
Thus Proposition~\ref{prop:simple} applies, and $S(Q)$ is simple.
Simplicity of  $U(\mf{sl}(n^\infty))/I(C)$ follows by  Theorem~\ref{thm:quivercat}.
\end{proof}

From this result, if $C$ is a Type I irreducible c.l.s. for $U(\mathfrak{sl}(n^{\infty}))$, that is, $C = \mathcal{A}^n_p$ or $C = \mathcal{B}^n_p$ for some $p \ge 0$, the ideal $I(C)$ is maximal. Therefore, $U(\mathfrak{sl}(n^{\infty}))$ has infinitely many maximal integrable ideals which arise from the Type I irreducible coherent local systems. This is in contrast to the finitary Lie algebra $\mathfrak{sl}(\infty)$, since in \cite{PP} it is shown that $U(\mathfrak{sl}(\infty))$ has a unique maximal integrable ideal.

The proof of Proposition~\ref{prop:simple} applies the Perron-Frobenius theorem to $A^n_d$.  
This gives some information about the eigenvalues and eigenvectors of Type I matrices:  each Type I matrix has  a unique positive eigenvector (i.e. an eigenvector all of whose entries are positive).  The associated eigenvalue is positive and of maximal modulus.  All other eigenvectors contain  a negative entry.  In the  next few results, we will improve this by completely describing the eigenvalues and eigenvectors of Type I matrices. 

First, we need a combinatorial result.

\begin{lemma}\label{subsets}
Let $d$ be a positive integer.  Then 
\[\sum\limits_{\substack{k_1, k_2, \dots k_n \\ k_1 + k_2 + \cdots + k_n = d}} \binom{d}{k_1} \binom{d - {k_1}}{k_2} \dots \binom{d-(k_1 + k_2 + \dots k_{n-2})}{k_{n-1}} = n^d. \]
\end{lemma}

\begin{proof}
We give a combinatorial proof. Suppose we want to count the number of ways to distribute $d$ elements into $n$ subsets. We could proceed by going through the list of elements and assigning a subset to each--there will be $n$ choices for each element, hence we obtain the right hand side, $n^d$. On the other hand, we could first determine the size of each subset, say subset $i$ will have size $k_i$, and then choose $k_1$ elements to be in the first subset, $k_2$ elements to be in the second subset, etc. This yields the left hand side of the above equation and completes the proof.  
\end{proof}

The next result uses representation theory to compute the largest eigenvalue of the Type I matrix $A^n_d$.
If $\lambda$ is a partition of $d$, let $H_\lambda$ denote the corresponding irreducible representation of the symmetric group $S_d$.

\begin{lemma}\label{lemma:repthry}
$A^n_d$ has eigenvalue $n^d$ with a positive eigenvector. 
\end{lemma}

\begin{proof}
Suppose $V^{nk}$ is the standard representation for $\mathfrak{gl}(nk)$ and recall that Schur-Weyl duality gives the decomposition:
\begin{eqnarray*}
(V^{nk})^{\otimes d} \cong \bigoplus_{|\lambda| = d} H_{\lambda} \otimes V^{nk}_{\lambda} .
\end{eqnarray*}
Let  $m_{\lambda} = \dim H_{\lambda}$ and let $v_d = (m_{\lambda})_{|\lambda| = d}$. Then $v_d$ is a positive vector whose entries consist of the coefficients $m_{\lambda}$ for $|\lambda| = d$, written in reverse lexicographical order as before. We will show that $A^n_d v_d = n^d v_d$. We will do this by calculating the restriction of $(V^{nk})^{\otimes d}$ to $\mathfrak{gl}(k)$ in two different ways. 

First, we have:  
\begin{eqnarray*}
\restr{(V^{nk})^{\otimes d}}{\mf{gl}(k)} \cong \bigoplus_{|\lambda| = d} m_{\lambda} \restr{V^{nk}_{\lambda} }{\mf{gl}(k)} \cong \bigoplus_{|\lambda| = d} m_{\lambda} (\bigoplus_{|\mu| = d} a^n_{\lambda, \mu} V^k_{\mu}).
\end{eqnarray*}
Therefore, the coefficient of $V^k_{\mu}$ in this decomposition is $\sum_{|\lambda| = d} a^n_{\mu, \lambda} m_{\lambda}$. On the other hand, since $V^{nk} \cong \underbrace{V^k \oplus V^k \oplus \cdots \oplus V^k}_{n}$ as a $\mathfrak{gl}(k)$ module, we have:
\begin{multline*}
\restr{(V^{nk})^{\otimes d}}{\mf{gl}(k)} \cong (\underbrace{V^k \oplus V^k \oplus \cdots \oplus V^k}_{n})^{\otimes d} \\
 \cong \sum\limits_{\substack{k_1, k_2, \dots k_n \\ k_1 + k_2 + \cdots + k_n = d}} \binom{d}{k_1} \binom{d - {k_1}}{k_2} \dots \binom{d-(k_1 + k_2 + \dots k_{n-2})}{k_{n-1}} (V^k)^{\otimes k_1} \otimes (V^k)^{\otimes{k_2}} \cdots \otimes (V^k)^{\otimes{k_n}} \\
=  n^d (V^k)^{\otimes d} .
\end{multline*}
(The last equality follows from Lemma \ref{subsets}.) Thus we obtain:
\begin{eqnarray*}
\restr{(V^{nk})^{\otimes d} }{\mf{gl}(k)} \cong n^d (V^k)^{\otimes d} \cong  \bigoplus_{|\mu| = d} n^d m_{\mu} V^{k}_{\mu} .
\end{eqnarray*}

Comparing the coefficients of $V^{k}_{\mu}$, we see that $\sum_{|\lambda| = d} a^n_{\mu, \lambda} m_{\lambda} = n^d m_{\mu}$. The left hand side is the $\mu$-entry of $A^n_d v_d$, and so $v_d$ is a positive eigenvector for the eigenvalue $n^d$. 
\end{proof}

For example, the matrix $A^2_2 = \left( \begin{array}{cc} 3 & 1 \\ 1 & 3 \end{array} \right)$ has eigenvalue $2^2 = 4$ with eigenvector $(1,1)$, which corresponds to the well-known decomposition $(V^{2k})^{\otimes 2} \cong Sym^2(V^{2k}) \oplus \bigwedge^2(V^{2k}) = V^{2k}_{(2)} \oplus V^{2k}_{(1,1)}$. The matrix $A^2_3 = \left( \begin{array}{ccc} 4 & 2 & 0\\ 2 & 6 & 2 \\ 0 & 2 & 4 \end{array} \right)$ has eigenvalue $2^3 = 8$ with eigenvector $(1,2,1)$, which results from the decomposition $(V^{2k})^{\otimes 3} \cong Sym^3(V^{2k}) \oplus (V^{2k}_{(2,1,0)})^{\oplus 2} \oplus \bigwedge^3(V^{2k}) =  V^{2k}_{(3,0,0)} \oplus V^{2k}_{(2,1,0)} \oplus V^{2k}_{(2,1,0)} \oplus V^{2k}_{(1,1,1)}$.

There is another interpretation of the Type I matrices $A_d^n$. 
 If we restrict the representation $H_{\lambda}$ of $S_n$ to a module over the subgroup $S_l \times S_{d-l}$ for some $0 \le l \le d$ (where $S_l \times S_{d-l} \subseteq S_d$ in the canonical way), we obtain:
\begin{eqnarray*}
Res^{S_d}_{S_l \times S_{d-l}}(H_{\lambda}) \cong \bigoplus_{\beta_1, \beta_2} c^{\lambda}_{\beta_1, \beta_2} H_{\beta_1} \otimes H_{\beta_2}
\end{eqnarray*}
where the sum is over all partitions $\beta_1$  of $l$ and  $\beta_2$ of $d-l$. We can iterate this formula, and see that if we restrict $H_{\lambda}$ to a module over the subgroup $S_{l_1} \times S_{l_2} \times \dots \times S_{l_n}$, where $l_1 + l_2 + \dots l_n = d$ we obtain:
\begin{eqnarray*}
Res^{S_d}_{S_{l_1} \times S_{l_2} \times \dots \times S_{l_n}}(H_{\lambda}) \cong \bigoplus_{\beta_1, \dots, \beta_n} c^{\lambda}_{\beta_1, \beta_2, \dots, \beta_n} (H_{\beta_1} \otimes H_{\beta_2} \otimes \dots \otimes H_{\beta_n})
\end{eqnarray*}
where each $\beta_i$ is a partition of $\l_i$. Then we can induce back up to a module over $S_d$ by using the following formula for modules over $S_d$:
\begin{eqnarray*}
Ind^{S_d}_{S_{l} \times S_{d-l}}(H_{\beta_1} \otimes H_{\beta_2}) \cong \bigoplus_{\mu} c^{\mu}_{\beta_1, \beta_2} H_{\mu}.
\end{eqnarray*}

Iterating this formula, we see how to induce a module over $S_{l_1} \times S_{l_2} \times \dots \times S_{l_n}$ back to a module over $S_d$:
\begin{eqnarray*}
Ind^{S_d}_{S_{l_1} \times S_{l_2} \times \dots \times S_{l_n}}(H_{\beta_1} \otimes H_{\beta_2} \otimes \dots \otimes H_{\beta_n}) \cong \bigoplus_{\mu} c^{\mu}_{\beta_1, \beta_2, \dots, \beta_n} H_{\mu}.
\end{eqnarray*}

Let $\mathcal{F}_{l_1, \dots, l_n}$ denote the functor that applies restriction then induction, i.e. 
\[\mathcal{F}_{l_1, \dots, l_n}(H_{\lambda}) = Ind^{S_d}_{S_{l_1} \times S_{l_2} \times \dots \times S_{l_n}}(Res^{S_d}_{S_{l_1} \times S_{l_2} \times \dots \times S_{l_n}}(H_{\lambda})) .\]
If we sum over all functors $\mathcal{F}_{l_1, \dots, l_n}$ and apply \eqref{adef}, we obtain familiar coefficients:
\begin{eqnarray}\label{billy}
\sum_{l_1, \dots, l_n} \mathcal{F}_{l_1, \dots, l_n}(H_\lambda) = 
 \bigoplus_{\beta_1, \beta_2, \dots, \beta_n, \mu}  c^{\lambda}_{\beta_1, \beta_2, \dots, \beta_n}c^{\mu}_{\beta_1, \beta_2, \dots, \beta_n} H_{\mu} =
\bigoplus_{\mu} a^n_{\lambda, \mu} H_{\mu}.
\end{eqnarray}

We will use \eqref{billy} to show that the eigenvectors of the matrix $A^n_d$ are given by the character table of $S_d$. For example, for $A^2_2 = \left( \begin{array}{cc} 3 & 1 \\ 1 & 3 \end{array} \right)$, we can choose a basis of eigenvectors to be $(1,1)$ and $(1,-1)$, which are exactly the column vectors in the character table for $S_2$.

We first need a definition:  
Given $\sigma \in S_d$ and nonnegative integers $l_1, \dots, l_n$ with $\sum l_i = d$, an {\em $(\l_1, \dots, \l_n)$-realization} of $\sigma $ is an ordered choice of $n$ disjoint subsets $X_i \subseteq \{1, 2, \dots, d\}$ with $|X_i| = l_i$ such that each disjoint cycle of $\sigma$ is contained in $S_{X_i}$ for some $1 \le i \le n$.
For example, there are two different $(2,2)$-realizations of $\sigma = (12)(34)$: $\sigma \in S_{\{1,2\}} \times S_{\{3,4\}}$ and $\sigma \in S_{\{3,4\}} \times S_{\{1,2\}}$. Note that a realization contains more information than fixing a subgroup of $S_d$ since the order of the $S_{X_i}$ matters. 

We can count the total number of realizations of any 
$\sigma$ 
for all possible $(l_1, l_2, \dots, l_n)$ with $l_1 + l_2 + \cdots + l_n = d$. Suppose $\sigma$ is the product of $p$ disjoint cycles. Then any realization is obtained by assigning an element of $\{1, 2, \dots, n\}$ to each cycle. Therefore, there are $n^p$ realizations of $\sigma \in S_d$.

For example, if $\sigma = (12)(34) \in S_4$ and $n = 2$, we may assign the elements of each disjoint cycle to either $X_1$ or $X_2$. Thus there are $2^2 = 4$ realizations of 
$\sigma $ 
for all possible $l_1$ and $l_2$: $\sigma \in S_{\emptyset} \times S_{\{1,2,3,4\}}$, $\sigma \in S_{\{1,2\}} \times S_{\{3,4\}}$, $\sigma \in S_{\{3,4\}} \times S_{\{1,2\}}$, and $\sigma \in S_{\{1,2,3,4\}} \times S_{\emptyset}$. If $\sigma = (12)(34) \in S_4$ and $n = 3$, then we have 3 choices for each disjoint cycle, so there will be $3^2 = 9$ different $(l_1,l_2,l_3)$-realizations of $\sigma $ for all possible $l_1, l_2, l_3$.

\begin{lemma}\label{coset}
Let $G= S_{l_1} \times S_{l_2} \times \dots \times S_{l_n} \subseteq S_d$ and fix $\sigma \in S_d$. Then the number of cosets of $G$ fixed by $\sigma$ is the number 
of $(l_1, \dots l_n)$-realizations of $\sigma$.
\end{lemma}

\begin{proof}
We identify $G$ with $S_{\{1, 2, \dots, l_1\}} \times S_{\{l_1 + 1, \dots, l_1 + l_2\}} \times \cdots \times S_{\{(l_1 + l_2 + \dots + l_{n-1}) + 1, \dots, d\}}$. First suppose $sG$ is a coset of $G$ fixed by $\sigma$, that is, $\sigma sG= sG$. 
Then $\sigma \in sGs^{-1}$, and so $s$ induces a permutation on the subscripts of $S_{\{1, 2, \dots, l_1\}} \times S_{\{l_1 + 1, \dots, l_1 + l_2\}} \times \cdots \times S_{\{(l_1 + l_2 + \dots + l_{n-1}) + 1, \dots, d\}}$ and gives 
an $(l_1, \dots l_n)$-realization of $\sigma$.
This realization is independent of coset representative: if $sG = tG$, then 
$sGs^{-1} = tGt^{-1}$ and so $s, t$ give the same $(l_1, \dots, l_n)$-realization of $\sigma$.
Conversely, given some realization $\sigma \in S_{X_1} \times S_{X_2} \times \dots \times S_{X_n}$, we can find some element $s \in S_d$ which permutes the subscripts of $S_{\{1, 2, \dots, l_1\}} \times S_{\{l_1 + 1, \dots, l_1 + l_2\}} \times S_{\{(l_1 + l_2 + \dots + l_{n-1}) + 1, \dots, d\}}$ to obtain $S_{X_1} \times S_{X_2} \times \dots \times S_{X_n}$. Then $\sigma \in sGs^{-1}$ and so  $\sigma sG = sG$. If two permutations $s, t \in S_d$ give the same realization, then $st^{-1}G(ts^{-1}) = st^{-1}S_{\{1, 2, \dots, l_1\}}(ts^{-1}) \times st^{-1}S_{\{l_1 + 1, \dots, l_1 + l_2\}}(ts^{-1}) \times \cdots \times st^{-1}S_{\{(l_1 + l_2 + \dots + l_{n-1}) + 1, \dots, d\}}(ts^{-1}) = S_{\{1, 2, \dots, l_1\}} \times S_{\{l_1 + 1, \dots, l_1 + l_2\}} \times \cdots \times S_{\{(l_1 + l_2 + \dots + l_{n-1}) + 1, \dots, d\}}$, which implies that $st^{-1} \in G$ since the only elements which enact the trivial permutation on subscripts (taking order into account) are contained in $G$ itself. Therefore $sG = tG$. This shows that there is a one-to-one correspondence between realizations and cosets fixed by $\sigma$. 
\end{proof}

Recall that  $H_{\lambda}$ is the irreducible finite dimensional representation of $S_d$ determined by the partition $\lambda$ of $d$. 
Let $\chi_{\lambda}(\sigma)$ denote the associated character of $S_d$. 
We now prove the main result of this section.

\begin{theorem}\label{thm:character}
Suppose $\sigma = \sigma_1 \dots \sigma_p \in S_d$ is a product of $p$ disjoint cycles, and let $\chi_{\lambda}$ denote the character of the irreducible representation $H_{\lambda}$. Then the vector $(\chi_{\lambda}(\sigma))_{|\lambda| = d}$ is an eigenvector for $\mathcal{A}^n_d$ with eigenvalue $n^p$.
\end{theorem}

\begin{proof}
Fix $\sigma= \sigma_1 \dots \sigma_p \in S_d$ such that $\sigma$ is the product of $p$ disjoint cycles. We compute the character of $\sigma$ on $M = \bigoplus_{l_1, \dots, l_n} \mathcal{F}_{l_1, \dots, l_n}(H_\lambda)$ in two different ways. First, from \eqref{billy} we see that the character of $\sigma$ on $M$  is $\sum_{|\mu| = d} a^n_{\lambda, \mu} \chi_{\mu}(\sigma)$. 

We now compute the character in another way, using the formula for characters of induced modules. 
Let $G= S_{l_1} \times S_{l_2} \times \dots \times S_{l_n}$. Using the formula for induced characters, the character of $\sigma$ on  $\mathcal{F}_{l_1, \dots, l_n}(H_{\lambda})$ is:
\begin{eqnarray*}
\chi^{l_1, \dots, l_n}_{Ind}(\sigma) = \sum_{\sigma C = C} \chi_{\lambda}(s\sigma s^{-1}) = (\mbox{number of cosets of $G$ fixed by $\sigma$}) \chi_{\lambda}(\sigma) ,
\end{eqnarray*}
where the sum is over cosets $C = sG$. By Lemma \ref{coset}, the number of cosets of $G$ fixed by $\sigma$ is the number of 
$(l_1, \dots, l_n)$-realizations of $\sigma$.  (In particular, if no such realizations exist, the character of $\sigma$ on the induced module is zero.) Therefore, the character of $\sigma$ on $\sum_{l_1, \dots, l_n} \mathcal{F}_{l_1, \dots, l_n}(H_\lambda)$ is:
\begin{eqnarray*}
\sum_{l_1, \dots, l_n} \chi^{l_1, \dots, l_n}_{Ind}(\sigma) = (\sum_{l_1, \dots, l_n} \mbox{number of $(l_1, \dots, l_n)$-realizations of $\sigma$} ) \chi_{\lambda}(\sigma).
\end{eqnarray*}
Write this as  $N  (\chi_{\lambda}(\sigma))$.

We can count the number $N$ in another way. Any realization of $\sigma$ is obtained by distributing the $p$ disjoint cycles of $\sigma$ into $n$ subsets. Therefore, by Lemma \ref{subsets}, we conclude that $N = n^p$. 
\end{proof}

\begin{remark}\label{rem:one}
Fix $k, n \in \ZZ_{\geq 1}$ and let $V$ be the natural representation of $\mf{gl}(nk)$.
The module
\[ V^{\otimes d} \cong \bigoplus_{|\lambda|=d} H_\lambda \otimes V^{nk}_\lambda \cong  \bigoplus_{|\lambda|=d} \chi_\lambda(1) V^{nk}_\lambda\]
is, of course, an important representation of $\mf{gl}(nk)$.
If $\sigma \in S_d$, we may consider 
\beq \label{virtual}
\sum_{|\lambda| = d} \chi_\lambda(\sigma)\left[  V^{nk}_\lambda \right]
\eeq
as a  class in $K_0(\mf{gl}(nk))$; this class is non-effective if $\sigma \neq 1$.  For example, if $\sigma = (12) \in S_2$, then
\[\eqref{virtual}  = \left[ V^{nk}_{(2)} \right] - \left[V^{nk}_{(1,1)}\right].\]

By Theorem~\ref{thm:character}, \eqref{virtual} is an eigenvector of the map induced by restriction to $K_0(\mf{gl}(k))$, using the diagonal embedding of signature $(n,0,0)$.   
It would be interesting to see if the ``virtual representation'' \eqref{virtual} has other pleasing representation-theoretic properties.

\end{remark}

We now have a complete description of all eigenvalues and eigenvectors of $A^n_d$: the eigenvalues are $\{n^{p(\sigma)} \st  [\sigma] \in Cl(S_d)\}$, where $\sigma$ consists of  $p(\sigma)$ disjoint cycles and $Cl(S_d)$ denotes the set of conjugacy classes in $S_d$. Since there are $P(d)$ conjugacy classes in $S_d$ and $A^n_d$ is a $P(d) \times P(d)$ matrix, we have constructed a full set of eigenvalues and eigenvectors.

For example, consider $A^2_2 = \left( \begin{array}{cc} 3 & 1 \\ 1 & 3 \end{array} \right)$, that is, $n = d = 2$. We know there are two elements (and two conjugacy classes) of $S_2$ given by the identity, $id = (1)(2)$, and the transposition $(12)$. Therefore the eigenvalues of $A^2_2$ will be $2^2$ with eigenvector $(\chi_{\lambda}(id))_{|\lambda| = 2} = (1,1)$ and $2^1 = 2$ with eigenvector $(\chi_{\lambda}(12))_{|\lambda| = 2} = (1,-1)$.

As another example, consider the matrix $A^2_3 = \left( \begin{array}{ccc} 4 & 2 & 0\\ 2 & 6 & 2 \\ 0 & 2 & 4 \end{array} \right)$ where $n = 2$ and $d = 3$. There are 3 conjugacy classes in $S_3$ with representatives given by the identity, $1$, $\sigma_1= (12)(3)$, and $\sigma_2 = (123)$. Therefore, the eigenvalues for $A^2_3$ will be $2^3 = 8$ with eigenvector $(\chi_{\lambda}(1))_{|\lambda| = 3} = (1,2,1)$, $2^2 = 4$ with eigenvector $(\chi_{\lambda}((12)(3))_{|\lambda| = 3} = (1,0,-1)$, and $2^1 = 2$ with eigenvector $(\chi_{\lambda}(123))_{|\lambda| = 3} = (1,-1,1)$.

Note that Theorem \ref{thm:character} provides a new proof of Lemma \ref{lemma:repthry}: there is only one $\sigma \in S_d$ which is the product of $d$ disjoint cycles, that is, $\sigma = 1$. Then the eigenvalue $n^d$ will occur with multiplicity 1 with positive eigenvector $(\chi_{\lambda}(1))_{|\lambda| = d} = (\dim H_{\lambda})_{|\lambda| = d} = (m_{\lambda})_{|\lambda| = d}$. Using the orthogonality relations in the character table of $S_d$, it is also clear that all other eigenvectors will have at least one nonpositive entry, and that the matrix $A^n_d$ is diagonalizable. 
This agrees with the remarks after Proposition~\ref{prop:simple2}.

Since the eigenvalues and eigenvectors of $A^n_d$ are now determined, we obtain a nice closed form for $A^n_d$. List the partitions of $d$ in reverse lexicographical order, as before. There is a bijection between partitions of $d$ and conjugacy classes of $S_d$: a partition $\lambda = (\lambda_1, \lambda_2, \dots, \lambda_p, 0, \dots, 0)$ is associated to the conjugacy class of $\sigma_{\lambda}= \sigma_1\sigma_2 \cdots \sigma_p$, where the $\sigma_i$ are disjoint cycles and $\sigma_i$ is a cycle of length $\lambda_i$. There is also a bijection between partitions of $d$ and irreducible representations of $S_d$, given by $\lambda \mapsto H_{\lambda}$. We define a $P(d) \times P(d)$ matrix $\X$ in the following way: the entry belonging to row $\lambda$ and column $\mu$ will be $\chi_{\lambda}(\sigma_{\mu})$, the character of $\sigma_{\mu}$ on the irreducible representation $H_{\lambda}$. The column of $\X$ associated to $\sigma_{\mu}$ will be $(\chi_{\lambda}(\sigma_{\mu}))_{|\lambda| = d}$, which is an eigenvector of $A^n_d$. Therefore, we have the following:

\begin{corollary}
Fix $d$ and let $\X$ be the $P(d) \times P(d)$ matrix described above, which is the character table of $S_d$. Then $A^n_d = \X D \X^{-1}$, where $D$ is a diagonal matrix whose diagonal entries are powers of $n$.
\end{corollary}
The entries of $D$ are given explictly by the eigenvalues in Theorem~\ref{thm:character}.

For example, for $d=2$ we have $\X = \begin{pmatrix} 1 & 1 \\ 1 & -1\end{pmatrix}$.
This gives that 
\[ A^n_2=  \X \begin{pmatrix} n^2 & \\ & n \end{pmatrix} \X^{-1} = 
\begin{pmatrix} \frac{n^2+n}{2} & \frac{n^2-n}{2} \\ \frac{n^2-n}{2} & \frac{n^2+n}{2} \end{pmatrix},\]
and we obtain the quivers in \eqref{iquiver}.

\section{Type II quivers}\label{TYPEII}

 In this section, we compute the eigenvalues of the Type II matrices for $\mathfrak{sl}(n^\infty)$ and the eigenvector with maximal eigenvalue, although our results  are not as complete as those for Type I matrices. 
Let $C^n_{p,q}$ denote the matrix for the coherent local system $\mathcal{C}^n_{p,q}$ of $\mathfrak{gl}(n^{\infty})$, where the entry belonging to row $(\lambda, \mu)$ and column $(\lambda', \mu')$ is the multiplicity of $V^{n^k}_{\lambda, \mu}$ in $V^{n^{k+1}}_{\lambda', \mu'}$ for $k \gg 0$.

We begin with an example: we have  $C^2_{1,1} = \left( \begin{array}{cc} 4 & 0 \\ 3 & 1 \end{array} \right)$. 
This follows from the following calculations. 
When we restrict the adjoint representation of $\mathfrak{gl}(2k)$ to $\mathfrak{gl}(k)$, we obtain:
\[
V^{2k}_{(1),(1)} \cong 4V^{k}_{(1),(1)} \oplus 3V^{k}_{(0),(0)}.\]
Further, the trivial representation of $\mathfrak{gl}(2k)$ remains trivial when we restrict to $\mathfrak{gl}(k)$, and we obtain that 
\[ \quad 
V^{2k}_{(0),(0)} \cong V^{k}_{(0),(0)}.
\]
 To translate this from $\mathfrak{gl}(k)$ to $\mathfrak{sl}(k)$, we note that the weight ${(1),(1)} = (1, 0, \dots, 0, -1)$ becomes $(2, 1, 1, \dots, 1, 0)$ over $\mathfrak{sl}(k)$ and ${(0),(0)} = (0,0, \dots, 0)$ remains the same. 

As another example,
\begin{eqnarray*}
C^{2}_{2,2} = \left( \begin{array}{cccccc} 9 & 3 & 3 & 1 & 0 & 0\\ 
3 & 9 & 1 & 3 & 0 & 0\\
3 & 1 & 9 & 3 & 0 & 0\\
1 & 3 & 3 & 9 & 0 & 0\\
12 & 12 & 12 & 12 & 4 & 0\\
6 & 3 & 3 & 6 & 3 & 1\end{array} \right) .
\end{eqnarray*}
This matrix is no longer symmetric as in the Type I case. Each column gives the decomposition of $V^{2k}_{\lambda, \mu}$ into $\mathfrak{gl}_k$ modules where the weights are ordered as follows: $\{(2), (2)\}$, $\{(2), (1,1)\}$, $\{(1,1), (2)\}$, $\{(1,1), (1,1)\}$, $\{(1), (1)\}$, $\{(0), (0)\}$. For example, the entries in the first column follow from the decomposition:
\begin{eqnarray*}
V^{2k}_{(2),(2)} \cong 9V^{k}_{(2),(2)} \oplus 3V^{k}_{(2),(1,1)} \oplus 3V^k_{(1,1),(2)} \oplus V^k_{(1,1),(1,1)} \oplus 12V^k_{(1),(1)} \oplus 6V^k_{(0),(0)}.
\end{eqnarray*}

Also note that $C^2_{1,1}$ is contained in the lower right corner of $C^2_{2,2}$. This will always be the case: namely, $C^n_{p,q}$ will always contain $C^n_{p-1,q-1}$ in its lower right corner, and the entries of $C^n_{p,q}$ in the rows above $C^n_{p-1,q-1}$ will all be zero. We also note that for Type II matrices, the length of the $\mathfrak{sl}$ weights which will appear are unbounded.

It is easy to see from the branching laws in Section \ref{BRANCH} that the Type II matrices are all block lower triangular.  Write $C^n_{p,q}$ as a matrix of blocks as follows:

\begin{eqnarray}\label{blocks}
C^n_{p,q} = \left( \begin{array}{ccccc} (C^n_{p,q})^{p,q} &  0  & 0 & \ldots & 0\\
B_{2,1}  & (C^n_{p,q})^{p-1,q-1} & 0  & \ldots & 0\\
B_{3,1}  & B_{3,2} & (C^n_{p,q})^{p-2,q-2}  & \ldots & 0\\
\vdots & \vdots & \vdots & \ddots & \vdots\\
B_{l,1} & B_{l,2} & B_{l,3} &\ldots & (C^n_{p,q})^{p-l+1,q-l+1}
 \end{array} \right) .
\end{eqnarray}
where $l = \min\{p, q\}$. The diagonal block $(C^n_{p,q})^{p',q'}$  gives branching multiplicities for pairs of partitions $(\lambda, \mu)$ and $(\lambda', \mu')$ with $|\lambda| = |\lambda'| = p' \le p$ and $|\mu| = |\mu'| = q' \le q$, where $p' - q' = p - q$. For example, $C^2_{2,2}$ has three diagonal blocks: $(C^2_{2,2})^{0,0} = (1)$, $(C^2_{2,2})^{1,1} = (4)$, and

\begin{eqnarray*}
(C^2_{2,2})^{2,2} = 
\left( \begin{array}{cccc} 9 & 3 & 3 & 1 \\ 
3 & 9 & 1 & 3\\
3 & 1 & 9 & 3\\
1 & 3 & 3 & 9\end{array} \right) .
\end{eqnarray*}

The next proposition calculates the diagonal blocks of $C^n_{p,q}$.

\begin{proposition}
We have $(C^n_{p,q})^{p',q'} = A^n_{p'} \otimes A^n_{q'}$, where $\otimes$ denotes the Kronecker product of matrices. 
\end{proposition}

\begin{proof}
Recall the following branching rule for $\mathfrak{gl}(k) \hookrightarrow \mathfrak{gl}(nk)$ given by Proposition \ref{branchgln}:
\begin{eqnarray*}
\restr{V^{nk}_{\lambda, \mu}}{\mathfrak{gl}(k)}\cong \bigoplus\limits_{\substack{\beta_1^+, \beta_2^+, \dots, \beta_n^+ \\ \beta_1^-, \beta_2^-, \dots, \beta_n^- \\ \lambda', \mu'}} C^{(\lambda, \mu)}_{(\beta_1^+, \beta_2^+, \dots, \beta_n^+), (\beta_1^-, \beta_2^-, \dots, \beta_n^-)}D^{(\lambda', \mu')}_{(\beta_1^+, \beta_2^+, \dots, \beta_n^+), (\beta_1^-, \beta_2^-, \dots, \beta_n^-)} V^k_{\lambda', \mu'}, 
\end{eqnarray*}
where 
$C^{(\lambda, \mu)}_{(\beta_1^+, \beta_2^+, \dots, \beta_n^+), (\beta_1^-, \beta_2^-, \dots, \beta_n^-)} $
and
$D^{(\lambda', \mu')}_{(\beta_1^+, \beta_2^+, \dots, \beta_n^+), (\beta_1^-, \beta_2^-, \dots, \beta_n^-)} $
are defined in that result.

Recall that if $C^{(\lambda, \mu)}_{(\beta_1^+, \beta_2^+, \dots, \beta_n^+), (\beta_1^-, \beta_2^-, \dots, \beta_n^-)} \ne 0$ and $D^{(\lambda', \mu')}_{(\beta_1^+, \beta_2^+, \dots, \beta_n^+), (\beta_1^-, \beta_2^-, \dots, \beta_n^-)} \ne 0$ we have:
\begin{eqnarray*}
|\lambda'| \le |\beta_1^+| + |\beta_2^+| + \dots + |\beta_n^+| \le |\lambda| \quad \mbox{ and } \\
|\mu'| \le |\beta_1^-| + |\beta_2^-| + \dots + |\beta_n^-| \le |\mu| .
\end{eqnarray*}

We wish to compute the multiplicity of $V^k_{(\lambda', \mu')}$ in $V^{nk}_{(\lambda, \mu)}$ where $|\lambda| = |\lambda'| = p'$ and $|\mu| = |\mu'| =q'$. 

Therefore, $|\lambda'| = |\beta_1^+| + |\beta_2^+| + \dots + |\beta_n^+| = |\lambda|$ and $|\mu'| = |\beta_1^-| + |\beta_2^-| + \dots + |\beta_n^-| = |\mu|$. Recall:
\begin{multline*}
C^{(\lambda, \mu)}_{(\beta_1^+, \beta_2^+, \dots, \beta_n^+), (\beta_1^-, \beta_2^-, \dots, \beta_n^-)} = \\
\sum\limits_{\substack{\alpha_1^+, \alpha_2^+, \dots, \alpha_{n-2}^+ \\ \alpha_1^-, \alpha_2^-, \dots, \alpha_{n-2}^-}} c^{(\lambda, \mu)}_{(\alpha_1^+, \alpha_1^-), (\beta_1^+, \beta_1^-)}c^{(\alpha_1^+, \alpha_1^-)}_{(\alpha_2^+, \alpha_2^-), (\beta_2^+, \beta_2^-)}\cdots c^{(\alpha_{n-3}^+, \alpha_{n-3}^-)}_{(\alpha_{n-2}^+, \alpha_{n-2}^-), (\beta_{n-2}^+, \beta_{n-2}^-)}c^{(\alpha_{n-2}^+, \alpha_{n-2}^-)}_{(\beta_{n-1}^+, \beta_{n-1}^-), (\beta_{n}^+, \beta_{n}^-)}
\end{multline*}

Since $|\lambda| = |\beta_1^+| + |\beta_2^+| + \dots + |\beta_n^+|$ we have that 
$|\lambda| = |\alpha^+_1| + |\beta^+_1|$ and
 \[ |\alpha_1^+| = |\alpha^+_2| + |\beta^+_2|, \dots, |\alpha_{n-3}^+| = |\alpha^+_{n-2}| + |\beta^+_{n-2}|, |\alpha_{n-2}^+| = |\beta^+_{n-1}| + |\beta^+_{n}|.\]
 Similarly, since $|\mu| = |\beta_1^-| + |\beta_2^-| + \dots + |\beta_n^-|$, we have that
  \[|\mu| = |\alpha^-_1| + |\beta^-_1|, \dots, |\alpha_{n-2}^-| = |\beta^-_{n-1}| + |\beta^-_{n}|.\] 

Recall the definition of $c^{(\lambda, \mu)}_{(\alpha^+, \alpha^-), (\beta^+, \beta^-)}$:
\begin{eqnarray*}
c^{(\lambda, \mu)}_{(\alpha^+, \alpha^-), (\beta^+, \beta^-)} = \sum_{\gamma^+, \gamma^-, \delta} c^{\lambda}_{\gamma^+, \delta}c^{\mu}_{\gamma^-, \delta}c^{\gamma^+}_{\alpha^+, \beta^+}c^{\gamma^-}_{\alpha^-, \beta^-} .
\end{eqnarray*}
When $|\lambda| = |\alpha^+| + |\beta^+|$ and $|\mu| = |\alpha^-| + |\beta^-|$, it must be that $\delta = 0$, $\gamma^+ = \lambda$ and $\gamma^- = \mu$. In this case $c^{(\lambda, \mu)}_{(\alpha^+, \alpha^-), (\beta^+, \beta^-)} = c^{\lambda}_{\alpha^+, \beta^+} c^{\mu}_{\alpha^-, \beta^-}$. 

Therefore, $C^{(\lambda, \mu)}_{(\beta_1^+, \beta_2^+, \dots, \beta_n^+), (\beta_1^-, \beta_2^-, \dots, \beta_n^-)}$ reduces to:
\begin{eqnarray*}
\sum\limits_{\substack{\alpha_1^+, \alpha_2^+, \dots, \alpha_{n-2}^+ \\ \alpha_1^-, \alpha_2^-, \dots, \alpha_{n-2}^-}} c^{\lambda}_{\alpha_1^+, \beta_1^+} c^{\mu}_{\alpha_1^-, \beta_1^-} c^{\alpha_1^+}_{\alpha_2^+, \beta_2^+} c^{\alpha^-_1}_{\alpha_2^-, \beta_2^-} \cdots c^{\alpha_{n-2}^+}_{\beta_{n-1}^+, \beta_n^+} c^{\alpha^-_{n-2}}_{\beta_{n-1}^-, \beta_n^-} = c^{\lambda}_{\beta^+_1, \beta^+_2, \dots, \beta^+_n} c^{\mu}_{\beta^-_1, \beta^-_2, \dots, \beta^-_n}.
\end{eqnarray*}
Similarly, recall that:
\begin{eqnarray*}
d^{(\lambda', \mu')}_{(\alpha^+, \alpha^-), (\beta^+, \beta^-)} = \sum\limits_{\substack{\alpha_1, \alpha_2, \beta_1, \beta_2 \\ \gamma_1, \gamma_2}} c^{\alpha^+}_{\alpha_1, \gamma_1}c^{\beta^-}_{\gamma_1, \beta_2}c^{\alpha^-}_{\beta_1, \gamma_2}c^{\beta^+}_{\gamma_2, \alpha_2}c^{\lambda'}_{\alpha_1, \alpha_2}c^{\mu'}_{\beta_1, \beta_2} .
\end{eqnarray*}

When $|\lambda'| = |\alpha^+| + |\beta^+|$ and $|\mu'| = |\alpha^-| + |\beta^-|$, $d^{(\lambda', \mu')}_{(\alpha^+, \alpha^-), (\beta^+, \beta^-)} = c^{\lambda'}_{\alpha^+, \beta^+} c^{\mu'}_{\alpha^-, \beta^-}$. Therefore, since $|\lambda'| = |\beta_1^+| + |\beta_2^+| + \dots + |\beta_n^+|$ and $|\mu'| = |\beta_1^-| + |\beta_2^-| + \dots + |\beta_n^-|$, the coefficient $D^{(\lambda', \mu')}_{(\beta_1^+, \beta_2^+, \dots, \beta_n^+), (\beta_1^-, \beta_2^-, \dots, \beta_n^-)}$ reduces to:
\begin{eqnarray*}
\sum\limits_{\substack{\alpha_1^+, \alpha_2^+, \dots, \alpha_{n-2}^+ \\ \alpha_1^-, \alpha_2^-, \dots, \alpha_{n-2}^-}} c^{\alpha_1^+}_{\beta_1^+,\beta_2^+} c^{\alpha_1^-}_{ \beta_1^-, \beta_2^-}
c^{\alpha_2^+}_{\alpha_1^+,\beta_3^+} c^{\alpha_2^-}_{\alpha_1^-, \beta_3^-} \cdots 
c^{\lambda'}_{\alpha_{n-2}^+, \beta_{n}^+}
c^{\mu'}_{\alpha_{n-2}^-, \beta_{n}^-} 
= c^{\lambda'}_{\beta_n^+, \beta_{n-1}^+, \dots, \beta_1^+}c^{\mu'}_{\beta_n^-, \beta_{n-1}^-, \dots, \beta_1^-} .
\end{eqnarray*}

Note that the order of the subscripts in the generalized Littlewood Richardson coefficients does not matter; that is, $c^{\lambda'}_{\beta_n^+, \beta_{n-1}^+, \dots, \beta_1^+} = c^{\lambda'}_{\beta^+_1, \beta^+_2, \dots, \beta^+_n}$ and $c^{\mu'}_{\beta_n^-, \beta_{n-1}^-, \dots, \beta_1^-}
= c^{\mu'}_{\beta^-_1, \beta^-_2, \dots, \beta^-_n}$. This follows from the interpretation of these coefficients in terms of induction and restriction functors in Section \ref{type1}.

Therefore the multiplicity of $V^k_{(\lambda', \mu')}$ in $V^{nk}_{(\lambda, \mu)}$ is:

\begin{eqnarray*}
\sum\limits_{\substack{\beta_1^+, \beta_2^+, \dots, \beta_n^+ \\ \beta_1^-, \beta_2^-, \dots, \beta_n^-}} c^{\lambda}_{\beta^+_1, \beta^+_2, \dots, \beta^+_n} c^{\mu}_{\beta^-_1, \beta^-_2, \dots, \beta^-_n} c^{\lambda'}_{\beta^+_1, \beta^+_2, \dots, \beta^+_n} c^{\mu'}_{\beta^-_1, \beta^-_2, \dots, \beta^-_n} 
\\ = (\sum_{\beta_1^+, \beta_2^+, \dots, \beta_n^+} c^{\lambda}_{\beta^+_1, \beta^+_2, \dots, \beta^+_n} c^{\lambda'}_{\beta^+_1, \beta^+_2, \dots, \beta^+_n}) (\sum_{\beta_1^-, \beta_2^-, \dots, \beta_n^-} c^{\mu}_{\beta^-_1, \beta^-_2, \dots, \beta^-_n} c^{\mu'}_{\beta^-_1, \beta^-_2, \dots, \beta^-_n}) \\ 
= a^n_{\lambda, \lambda'} a^n_{\mu, \mu'} .
\end{eqnarray*}

This is the entry in $(C^n_{p,q})^{p',q'}$ belonging to row $(\lambda, \mu)$ and column $(\lambda', \mu')$, but we also see this is the corresponding entry in the product $A^n_{p'} \otimes A^n_{q'}$.
\end{proof}

This allows us to compute the eigenvalues of $C^n_{p,q}$ using the results of Section \ref{type1}. We use the following well-known results about eigenvalues of matrices.

\begin{lemma}\label{blockeigenvalues}
Suppose $C$ is an $d \times d$ block lower triangular matrix with diagonal blocks $C_1, \dots, C_l$. Then the eigenvalues of $C$ are the combined eigenvalues from all of the blocks $C_1, \dots, C_l$. 
\end{lemma}

\begin{lemma}\label{kronecker}
Suppose $A$ and $B$ are square matrices. Suppose $\lambda$ is an eigenvalue of $A$ with eigenvector $x$ and $\mu$ is an eigenvalue of $B$ with eigenvector $y$. Then $\lambda\mu$ is an eigenvalue of the Kronecker product $A \otimes B$ with eigenvector $x \otimes y$.
\end{lemma}

\begin{theorem}\label{thm:typeIIevals}
The eigenvalues of $C^n_{p,q}$ are $\{1, n, \dots, n^{p+q}\}$, and $n^{p+q}$ has multiplicity 1.
\end{theorem}

\begin{proof}
Since $C^n_{p,q}$ is block lower triangular, by Lemma \ref{blockeigenvalues} its eigenvalues will be the set of eigenvalues which appear in all of its diagonal blocks, i.e. $\{ \mbox{eigenvalues of } (C^n_{p,q})^{p', q'} \st p' \le p, q' \le q, p-q = p'-q'\}$. Since $(C^n_{p,q})^{p', q'} = A^n_{p'} \otimes A^n_{q'}$, its set of eigenvalues will be $\{ \lambda_i \mu_j \}$, where $\lambda_1, \lambda_2, \dots, \lambda_M$ are the eigenvalues of $A^n_{p'}$ and $\mu_1, \mu_2, \dots, \mu_N$ are the eigenvalues of $A^n_{q'}$. From section \ref{type1}, we know that the eigenvalues of $A^n_{d}$ (with multiplicity) are $\{n^{p(\sigma)} \st [\sigma] \in Cl(S_d)\}$, and $n^d$ is the unique maximal eigenvalue of $A^n_d$. Thus, the eigenvalues of $C^n_{p,q}$ are all powers of $n$, and $n^{p+q}$ is the maximal eigenvalue of $C^n_{p,q}$. It is also clear that the multiplicity of $n^{p+q}$ is $1$.
\end{proof}

While we have a complete description of the eigenvalues for the Type II matrices, we do not have a classification of the eigenvectors. We conjecture that each Type II matrix is diagonalizable.
We do have the following result, which we state without proof:
\begin{proposition}\label{posgln}
The eigenvalue $n^{p+q}$ of $C^n_{p,q}$ has an eigenvector consisting of positive entries, and it is the only positive eigenvector of $C^n_{p,q}$.
\end{proposition}

While each $C^n_{p,q}$ will have a unique positive eigenvector, it is also apparent that there will be other eigenvectors with non-negative entries, and therefore not all of the associated ideals will be maximal. This is also clear from Proposition \ref{clsgl}: The irreducible c.l.s. $\mathcal{C}^n_{p,q}$ can be partitioned into sets, $\mathcal{C}^n_m$, indexed by the integers, where $\mathcal{C}^n_m = \{ \mathcal{C}^n_{p,q} \st p - q = m\}$. Each $\mathcal{C}^n_m$ gives a proper ascending chain of irreducible c.l.s. with a unique minimal element. For example, $\mathcal{C}^n_0$ will consist of the proper ascending chain $\mathcal{C}^n_{0,0} \subseteq \mathcal{C}^n_{1,1} \subseteq \mathcal{C}^n_{2,2} \subseteq \dots$. Taking annihilators $I(C)$ for $C \in \mathcal{C}^n_m$, we obtain a proper descending chain of prime ideals. This shows that only the annihilator of the minimal c.l.s. in $\mathcal{C}^n_m$ can be maximal. For $m \ge 0$, the minimal c.l.s. in $\mathcal{C}^n_m$ is $\mathcal{C}^n_{m, 0} = \mathcal{A}^n_m$. For $m < 0$, the minimal c.l.s. in $\mathcal{C}^n_m$ is $\mathcal{C}^n_{0, m} = \mathcal{B}^n_m$. In both cases, we know the associated ideal is maximal by our analysis of Type I matrices in the previous section. Therefore, we see that the prime ideal $I(C)$ associated to a Type II c.l.s. $C$ will be maximal only when the c.l.s. $C$ is minimal, which occurs only when the c.l.s. $C$ is of Type I.

\section{Quivers for $\mathfrak{sp}(n^{\infty})$ and $\mathfrak{so}(n^{\infty})$}\label{SPQUIVERS}

The matrices for the irreducible coherent local systems of $\mathfrak{sp}(n^{\infty})$ and $\mathfrak{so}(n^{\infty})$ exhibit behavior similar to the Type II matrices of $\mathfrak{sl}(n^\infty)$. Let $D^n_p$ denote the matrix for the coherent local system $\mathcal{D}^n_p$ of $\mathfrak{sp}(n^{\infty})$ (which is only defined if $n$ is even).
Let $E^n_p$ denote the matrix for the coherent local system $\mathcal{E}^n_p$ of $\mathfrak{so}(n^{\infty})$. 
For example, $D^{2}_2 = \left( \begin{array}{ccc} 
3 & 1 & 0 \\
1 & 3 & 0\\
1 & 2 & 1\end{array} \right)$ and $D^{2}_3 = \left( \begin{array}{cccc} 
4 & 2 & 0 & 0\\
2 & 6 & 2 & 0\\
0 & 2 & 4 & 0 \\
2 & 6 & 4 & 2\end{array} \right)$.
The matrices for $\mathfrak{so}(n^{\infty})$ are similar: $E^{2}_2 = \left( \begin{array}{ccc} 
3 & 1 & 0 \\
1 & 3 & 0\\
2 & 1 & 1\end{array} \right)$ and $E^{2}_3 = \left( \begin{array}{cccc} 
4 & 2 & 0 & 0\\
2 & 6 & 2 & 0\\
0 & 2 & 4 & 0 \\
4 & 6 & 2 & 2\end{array} \right)$.
From the examples above, the diagonal blocks of $D^n_p$ and $E^n_p$ seem to be Type I matrices.  We will see this is in fact true.

Note that by the branching laws in Section \ref{BRANCH}, both $D^n_p$ and $E^n_p$ are block lower triangular.  Suppose $X \in \{D, E\}$. We can write $X^n_p$ as a matrix of blocks as follows:

\begin{eqnarray}\label{blockssp}
X^n_p = \left( \begin{array}{ccccc} (X^n_p)^p &  0  & 0 & \ldots & 0\\
X_{2,1}  & (X^n_p)^{p-2} & 0  & \ldots & 0\\
X_{3,1}  & X_{3,2} & (X^n_p)^{p-4}  & \ldots & 0\\
\vdots & \vdots & \vdots & \ddots & \vdots\\
X_{l,1} & X_{l,2} & X_{l,3} &\ldots & (X^n_p)^{p-2l}
 \end{array} \right)
\end{eqnarray}
where $l = \lfloor \frac{p}{2} \rfloor $. Each diagonal block $(X^n_p)^{p'}$ is a square matrix which gives the branching multiplicities for partitions $\lambda$ and $\lambda'$ with $|\lambda| = |\lambda'| = p'$, where $p' \le p$ and $p' \equiv p \mod 2$. We call such a block $(D^n_p)^{p'}$ for $\mathfrak{sp}(n^{\infty})$ (respectively, $(E^n_p)^{p'}$ for $\mathfrak{so}(n^{\infty})$).

\begin{lemma}\label{blockspn}
The diagonal block $(E^n_p)^{p'} = A^n_{p'}$.
This is further equal to $(D^n_p)^{p'} $ if $n$ is even. 

\end{lemma}

\begin{proof}
We provide a proof for $(D^n_p)^{p'}$ and note that the proof for $(E^n_p)^{p'}$ is exactly the same. Fix partitions $\lambda$ and $\lambda'$ with $|\lambda| = |\lambda'| = p'$. The entry of $(D^n_p)^{p'}$ corresponding to row $\lambda'$ and column $\lambda$ will be the multiplicity of $U^{n^k}_{\lambda'}$ in $U^{n^{k+1}}_{\lambda}$ for $k \gg 0$. Recall the following branching rule (Proposition \ref{branchspn}):

\begin{eqnarray*}
\restr{U^{2nk}_{\lambda}}{\mathfrak{sp}(2k)} \cong \bigoplus_{\beta_1, \beta_2, \dots, \beta_n, \lambda'} E^{\lambda}_{\beta_1, \beta_2, \dots, \beta_n} F^{\lambda'}_{\beta_1, \beta_2, \dots, \beta_n} U^{2k}_{\lambda'}
\end{eqnarray*}
where:
\begin{align*} E^{\lambda}_{\beta_1, \beta_2, \dots, \beta_n} &  = \sum_{\alpha_1, \alpha_2, \dots, \alpha_{n-2}} e^{\lambda}_{\alpha_1, \beta_1} e^{\alpha_1}_{\beta_2, \alpha_2} \cdots e^{\alpha_{n-3}}_{\beta_{n-2}, \alpha_{n-2}}e^{\alpha_{n-2}}_{\beta_{n-1}, \beta_n}  \quad \mbox{ and } \\
F^{\lambda'}_{\beta_1, \beta_2, \dots, \beta_n} & = \sum_{\alpha_1, \alpha_2, \dots, \alpha_{n-2}} f^{\alpha_1}_{\beta_1, \beta_2} f^{\alpha_2}_{\alpha_1, \beta_3} \cdots f^{\alpha_{n-2}}_{\alpha_{n-3}, \beta_{n-1}} f^{\lambda'}_{\alpha_{n-2}, \beta_n} .
\end{align*}

Since $|\lambda'| = |\beta_1| + |\beta_2| + \dots + |\beta_n| = |\lambda|$, we must have that $|\lambda| = |\alpha_1| + |\beta_1|$, $|\alpha_1|= |\alpha_2| + |\beta_2|$, $\dots$, $|\alpha_{n-2}| = |\beta_{n-1}| + |\beta_n|$ in $E^{\lambda}_{\beta_1, \beta_2, \dots, \beta_n}$ and $|\alpha_1| = |\beta_1| + |\beta_2|$, $|\alpha_2 | = |\alpha_1| + |\beta_3|$, $\dots$, $|\lambda'| = |\alpha_n-2| + |\beta_n|$ in $F^{\lambda'}_{\beta_1, \beta_2, \dots, \beta_n}$. 

Recall the following definitions:
\[ e^{\lambda}_{\mu, \nu} = \sum_{\delta, \gamma} c^{\gamma}_{\mu, \nu} c^{\lambda}_{\gamma, (2\delta)^{T}}, \quad 
 f^{\lambda'}_{\mu, \nu} = \sum_{\alpha, \beta, \gamma'} c^{\lambda'}_{\alpha, \beta} c^{\mu}_{\alpha, \gamma'} c^{\nu}_{\beta, \gamma'} .\]

If $|\lambda'| = |\mu| + |\nu| = |\lambda'|$, then $e^{\lambda}_{\mu, \nu} = c^{\lambda}_{\mu, \nu}$ and $f^{\lambda'}_{\mu, \nu} = c^{\lambda'}_{\mu, \nu}$. Therefore, the entry of $(D^n_p)^{p'}$ in row $\lambda'$ and column $\lambda$ is:
\begin{multline*}E^{\lambda}_{\beta_1, \beta_2, \dots, \beta_n}F^{\lambda'}_{\beta_1, \beta_2, \dots, \beta_n} = \\
(\sum_{\alpha_1, \alpha_2, \dots, \alpha_{n-2}} c^{\lambda}_{\alpha_1, \beta_1} c^{\alpha_1}_{\beta_2, \alpha_2} \cdots c^{\alpha_{n-3}}_{\beta_{n-2}, \alpha_{n-2}}c^{\alpha_{n-2}}_{\beta_{n-1}, \beta_n})
(\sum_{\alpha_1, \alpha_2, \dots, \alpha_{n-2}} c^{\alpha_1}_{\beta_1, \beta_2} c^{\alpha_2}_{\alpha_1, \beta_3} \cdots c^{\alpha_{n-2}}_{\alpha_{n-3}, \beta_{n-1}} c^{\lambda'}_{\alpha_{n-2}, \beta_n}) = \\
(c^{\lambda}_{\beta_1, \beta_2, \dots, \beta_n}) (c^{\lambda'}_{\beta_n, \beta_{n-1}, \dots, \beta_1}) = (c^{\lambda}_{\beta_1, \beta_2, \dots, \beta_n}) (c^{\lambda'}_{\beta_1, \beta_2, \dots, \beta_n}) = a^n_{\lambda, \lambda'} .
\end{multline*}

This proves that $(D^n_p)^{p'} = A^n_{p'}$.
\end{proof}

As with Type II matrices, we have:

\begin{theorem}\label{thm:DEevals}
The eigenvalues for $E^n_p$ and $D^n_p$ ($n$ even) are $\{1, n, \dots, n^p\}$.
The eigenvalue $n^p$ has multiplicity 1.
\end{theorem}

\begin{proof}
This follows from Lemmas \ref{blockeigenvalues} and \ref{blockspn}.
\end{proof}

 We now have a description of the eigenvalues of the matrices associated to irreducible coherent local systems for $\mathfrak{sp}(n^{\infty})$ and $\mathfrak{so}(n^{\infty})$, but again we do not have a classification of the eigenvectors. Similarly to the situation for Type II matrices, we do know that the following is true:

\begin{proposition}\label{posspn}
Suppose $X \in \{D,E\}$. The eigenvalue $n^p$ of $X^n_p$ has an eigenvector consisting of positive entries, and it is the only positive eigenvector of $X^n_p$.
\end{proposition}


As in the Type II case, while there is only one positive eigenvector, there is not a unique eigenvector consisting of nonnegative entries, and therefore we know that the associated ideals will not be maximal. More precisely, suppose $\mathcal{X} = \mathcal{D}$ or $\mathcal{E}$. From Proposition \ref{clssp}, the irreducible c.l.s. $\mathcal{X}^n_p$ can be partitioned into two disjoint sets $\mathcal{X}^n_0$, consisting of $\mathcal{X}^n_p$ for $p$ even, and $\mathcal{X}^n_1$, consisting of $\mathcal{X}^n_p$ for $p$ odd. In each case we have a proper ascending chain of irreducible c.l.s. with unique minimal element, and therefore we have a proper descending chain of prime ideals with unique maximal element. The minimal c.l.s. of $\mathcal{X}^n_0$ is the c.l.s. consisting only of trivial modules. Its associated ideal is the augmentation ideal, which is clearly maximal. The minimal c.l.s. of $\mathcal{X}^n_1$ is the c.l.s. obtained from any natural module, and its associated ideal will be maximal by Proposition \ref{1to1}. Therefore, $U(\mathfrak{sp}(n^{\infty}))$ and $U(\mathfrak{so}(n^{\infty}))$ each will have two maximal ideals. This is in contrast to the finitary case: from \cite{PP}, we know that $U(\mathfrak{sp}(\infty))$ has only one integrable maximal ideal, the augmentation ideal.

\section{Appendix: Examples}\label{APPENDIX}

In this appendix we provide examples of the adjacency matrices of the quivers corresponding to the irreducible coherent local systems for $\mathfrak{sl}(n^{\infty})$, $\mathfrak{sp}(n^{\infty})$, and $\mathfrak{so}(n^{\infty})$ when $n = 2$.

\begin{table}[H] 
\caption{Type I matrices for $\mathfrak{sl}(2^{\infty})$}
\label{t1}
\begin{minipage}[b]{0.45\linewidth}\centering
$$A^2_0 = (1)$$

$$A^2_1 = (2)$$

$$A^2_2 = \left( \begin{array}{cc} 3 & 1 \\ 1 & 3 \end{array} \right)$$

$$A^2_3 = \left( \begin{array}{ccc} 4 & 2 & 0\\ 2 & 6 & 2 \\ 0 & 2 & 4 \end{array} \right)$$

\end{minipage}
\hspace{0.5cm}
\begin{minipage}[b]{0.45\linewidth}
\centering
$$A^2_4 = \left( \begin{array}{ccccc} 5 & 3 & 1 & 0 & 0 \\ 3 & 9 & 3 & 4 & 0 \\ 1 & 3 & 6 & 3 & 1 \\ 0 & 4 & 3 & 9 & 3 \\ 0 & 0 & 1 & 3 & 5 \end{array} \right)$$

$$A^2_5 = \left( \begin{array}{ccccccc} 
6 & 4 & 2 & 0 & 0 & 0 & 0 \\ 
4 & 12 & 6 & 6 & 2 & 0 & 0 \\ 
2 & 6 & 12 & 6 & 6 & 2 & 0 \\ 
 0 & 6 & 6 & 14 & 6 & 6 & 0 \\
0 & 2 & 6 & 6 & 12 & 6 & 2 \\ 
0 & 0 & 2 & 6 & 6 & 12 & 4 \\ 
0 & 0 & 0 & 0 & 2 & 4 & 6 \end{array} \right)$$

\end{minipage}
\end{table}

In  Tables~\ref{t2} and \ref{t3} the blocks, as defined in Sections~\ref{TYPEII} and \ref{SPQUIVERS}, are outlined.

\begin{table}[H]
\caption{Type II matrices for $\mathfrak{sl}(2^{\infty})$}
\label{t2}
\begin{minipage}[b]{0.45\linewidth}\centering
$$C^2_{0,0} = (1)$$

$$C^2_{1,1} = \left( \begin{array}{c|c} 4 & 0 \\\hline 3 & 1 \end{array} \right)$$

$$C^{2}_{2,2} = \left( \begin{array}{cccc|c|c} 9 & 3 & 3 & 1 & 0 & 0\\ 
3 & 9 & 1 & 3 & 0 & 0\\
3 & 1 & 9 & 3 & 0 & 0\\
1 & 3 & 3 & 9 & 0 & 0\\
\hline 12 & 12 & 12 & 12 & 4 & 0\\
\hline 6 & 3 & 3 & 6 & 3 & 1\end{array} \right)$$

$$C^2_{1,0} = (2)$$

\end{minipage}
\hspace{0.5cm}
\begin{minipage}[b]{0.45\linewidth}
\centering

$$C^2_{2,1} = \left( \begin{array}{cc|c} 6 & 2 & 0 \\ 2 & 6 & 0 \\ \hline 6 & 6 & 2 \end{array} \right)$$

$$C^2_{3,1} = \left( \begin{array}{ccc|cc}  8 & 4 & 0 & 0 & 0 \\ 4 & 12 & 4 & 0 & 0 \\ 0 & 4 & 8 & 0 & 0 \\ \hline 9 & 12 & 3 & 3 & 1 \\ 3 & 12 & 9 & 1 & 3 \end{array} \right)$$

$$C^2_{4,1} = \left( \begin{array}{ccccc|ccc} 10 & 6 & 2 & 0 & 0 & 0 & 0 & 0 \\ 6 & 18 & 6 & 8 & 0 & 0 & 0 & 0 \\ 2 & 6 & 12 & 6 & 2 & 0 & 0 & 0 \\ 0 & 8 & 6 & 18 & 6 & 0 & 0 & 0 \\ 0 & 0 & 2 & 6 & 10 & 0 & 0 & 0 \\  \hline 12 & 18 & 6 & 6 & 0 & 4 & 2 & 0 \\ 6 & 24 & 18 & 24 & 6 & 2 & 6 & 2 \\ 0 & 6 & 6 & 18 & 12 & 0 & 2 & 4 \end{array} \right)$$
\end{minipage}
\end{table}

\begin{table}[H]
\caption{}
\label{t3}
\begin{minipage}[b]{0.45\linewidth}\centering

Matrices for $\mathfrak{sp}(2^{\infty})$

$$D_0^2 = (1)$$

$$D_1^2 = (2)$$

$$D^{2}_2 = \left( \begin{array}{cc|c} 
3 &1 & 0 \\
1 & 3 & 0 \\
\hline 1 & 2  & 1 \\
\end{array} \right)$$ 

$$D^{2}_3 = \left( \begin{array}{ccc|c} 
4 & 2 & 0 & 0\\
2 & 6 & 2 & 0\\
0 & 2 & 4 & 0 \\
\hline 2 & 6 & 4 & 2
\end{array} \right)$$

$$D^{2}_4 = \left( \begin{array}{ccccc|cc|c}
 5 & 3 & 1 & 0 & 0 & 0 & 0 & 0  \\
 3 & 9 & 3 & 4 & 0 & 0 & 0 & 0 \\
 1 & 3 & 6 & 3 & 1 & 0 & 0 & 0 \\
 0 & 4 & 3 & 9 & 3 & 0 & 0 & 0  \\
 0 & 0 & 1 & 3 & 5 & 0 & 0 & 0  \\
\hline  3 & 10 & 5 & 9 & 2 & 3 &1 & 0 \\
 1 & 6 & 7 & 11 & 6 & 1 & 3 & 0 \\
\hline  1 & 2 & 4 & 3 & 3 & 1 & 2  & 1 \\ 
\end{array} \right)$$
\end{minipage}
\hspace{0.5cm}
\begin{minipage}[b]{0.45\linewidth}
\centering

Matrices for $\mathfrak{so}(2^{\infty})$

$$E_0^2 = (1)$$

$$E_1^2 = (2)$$

$$E^{2}_2 = \left( \begin{array}{cc|c} 
3 & 1 & 0 \\
1 & 3 & 0\\
\hline 2 & 1 & 1\end{array} \right)$$ 

$$E^{2}_3 = \left( \begin{array}{ccc|c} 
4 & 2 & 0 & 0\\
2 & 6 & 2 & 0\\
0 & 2 & 4 & 0 \\
\hline 4 & 6 & 2 & 2\end{array} \right)$$

$$E^{2}_4 = \left( \begin{array}{ccccc|cc|c}
 5 & 3 & 1 & 0 & 0 & 0 & 0 & 0  \\
 3 & 9 & 3 & 4 & 0 & 0 & 0 & 0 \\
 1 & 3 & 6 & 3 & 1 & 0 & 0 & 0 \\
 0 & 4 & 3 & 9 & 3 & 0 & 0 & 0  \\
 0 & 0 & 1 & 3 & 5 & 0 & 0 & 0  \\
\hline  6 & 11 & 7 & 6 & 1 & 3 &1 & 0 \\
 2 & 9 & 5 & 10 & 3 & 1 & 3 & 0 \\
\hline  3 & 3 & 4 & 2 & 1 & 1 & 2  & 1 \\ 
\end{array} \right)$$
\end{minipage}
\end{table}


\begin{thebibliography}{19}

\bibitem{BZh}{A. A. Baranov, A.G. Zhilinskii. Diagonal direct limits of simple Lie algebras. {\em Commun. in Algebra} {\bf 27} (1999), 2749--2766.}

\bibitem{BGR}{L. G. Brown, P. Green, and M. A.  Rieffel. Stable isomorphism and strong Morita equivalence of $C\sp*$-algebras. {\em Pacific J. Math.} {\bf 71} (1977), no. 2, 349--363.}

\bibitem{DPS}{E. Dan-Cohen, I. Penkov, and V. Serganova. A Koszul category of representations of finitary Lie algebras, {\em Adv. Math} {\bf 289} (2016), 250--278.}

\bibitem{Elliot}{G. A. Elliott. On the classification of inductive limits of sequences of semisimple finite-dimensional algebras, {\em J. Algebra} {\bf 38} (1976) 29--44.}

\bibitem{FH}{W. Fulton, J. Harris. {\em Representation Theory: A First Course}. Springer: New York, 1991.}

\bibitem{Goodearl}{K. Goodearl.  {\em Von Neumann regular rings}. Monographs and Studies in Mathematics, no. 4, Pitman: London, 1979.}

\bibitem{Goodearl-Leavitt}{K. Goodearl. Leavitt path algebras and direct limits. {\em Rings, modules and representations}, 165--187,  Contemp. Math.  480. Amer. Math. Soc.: Providence, RI, 2009. }

\bibitem{Ho}{C. Holdaway.   Ext quivers and noncommutative projective schemes coming from path algebras of finite GK-dimension. Unpublished manuscript,  personal communication, 2015.}

\bibitem{HTW}{R. Howe, E. Tan, and J. Willenbring. Stable branching rules for classical symmetric pairs. {\em  Trans. Amer. Math. Soc.} {\bf 357}(2005), no. 4, 1601--1626.}

\bibitem{Hr}{E. Hristova. Branching Laws for Tensor Modules over Classical Locally Finite Lie algebras,  {\em J. Algebra} {\bf 397} (2014) 278--314.}


\bibitem{McD}{I. G. MacDonald. {\em Symmetric Functions and Hall Polynomials}. Oxford: New York, 1979.}

\bibitem{PP}{I. Penkov, A. Petukhov. On ideals in the enveloping algebra of a locally simple Lie algebra. {\em  International Math. Res. Notices}, (2015) Vol. 2015,  5196--5228.}


\bibitem{PSe1}{I. Penkov, V. Serganova. Categories of integrable $\mf{sl}(\infty)$-, $\mf{o}(\infty)$-, $\mf{sp}(\infty)$-modules.  {\em  Representation Theory and Mathematical Physics}, 335--357, Contemporary Mathematics 557. Amer. Math. Soc.: Providence, RI,  2011.}

\bibitem{PSe2}{I. Penkov, V. Serganova. Tensor representations of Mackey Lie algebras and their dense subalgebras. {\em Developments and Retrospectives in Lie Theory: Algebraic Methods}, 291--330. Developments in Mathematics, 38, Springer Verlag.}

\bibitem{PSt}{I. Penkov, K. Styrkas. Tensor representations of infinite-dimensional root-reductive Lie algebras. {\em Developments and Trends in Infinite-Dimensional Lie Theory}, 127--159. Progress in Mathematics 288, Birkh\"auser, 2011.}


\bibitem{S}{S. P. Smith. Category equivalences involving graded modules over path algebras of quivers. {\em  Adv. Math.} {\bf 230} (2012), no. 4-6, 1780--1810.}

\bibitem{Zh1}{A. G. Zhilinskii. Coherent systems of representations of inductive families of simple complex Lie algebras. (Russian) Preprint of Academy of Belarussian, SSR, ser. 48(438), Minsk, 1990.}

\bibitem{Zh2}{A. G. Zhilinskii. On the lattice of ideals in the universal enveloping algebra of a diagonal
Lie algebra. (Russian) Preprint, Minsk, 2011.}

\end{thebibliography}
\end{document}